%% file: main.tex
\documentclass{article}

\input{packages}
\input{commands}

\AtBeginDocument{%
  \crefname{thmt@dummyctr}{setting}{settings}%
  \Crefname{thmt@dummyctr}{Setting}{Settings}%
}

\title{The critical exponent of the Falconer functional for Anosov representations}
\author{Giorgos Stamatiou}
\date{\today}

\begin{document}

\maketitle
\begin{abstract}
    We study the critical exponent of the Falconer functional for projective Anosov representations with Lipschitz limit sets and its relationship to the Hausdorff dimension of the limit set.
    Extending a result of Pozzetti, Sambarino, and Wienhard \cite{pozzetti_anosov_2023}, we show, under density and irreducibility assumptions on the representation, that this exponent equals the Hausdorff dimension of the limit set.
    We also show that these assumptions are necessary for the strategy used in \cite{pozzetti_anosov_2023} and in the present paper. 
    To this end, we study the action of $\mathrm{SO}(p,p)$ on the space of maximal isotropic subspaces of $\mathbb R^{p,p}$ and provide examples of $\mathbf H^{p,q}$-convex-cocompact representations of lattices with virtual cohomological dimension $p$.
    Besides yielding counterexamples to seemingly innocent claims about irreducible representations in $\mathrm{SO}(p,q+1)$, these examples are of independent interest: they admit equivariant spacelike embeddings of Riemannian symmetric spaces into $\mathbf H^{p,q}$ on which uniform lattices act cocompactly.
\end{abstract}
\tableofcontents

\section{Introduction}
\subsection{Sullivan's theorem}
Let $\mathbf H_{\mathbb R}^n$ denote the real hyperbolic space of dimension $n$ and $\Gamma \leq \SO(1, n)$ be a convex-cocompact subgroup of isometries.
To the action $\Gamma \curvearrowright \mathbf H_{\mathbb R}^n$, we associate its critical exponent $\delta_\Gamma$, which is a dynamical invariant that expresses the exponential growth rate of the $\Gamma$-orbits in $\mathbf H_{\mathbb R}^n$.
Another geometric object associated with the action is the limit set $\Lambda_\Gamma$, consisting of the accumulation points of an orbit at infinity.
In \cite{sullivan1979density}, Sullivan proved that the critical exponent $\delta_\Gamma$ is equal to the Hausdorff dimension of the limit set $\Lambda_\Gamma$, establishing a stunning connection between the fractal geometry of the limit set and the dynamics of the group action:
\begin{theorem}[\cite{sullivan1979density}]\label{thm:sullivan}
    Let $\Gamma \leq \PSO(1,n)$ be convex-cocompact.
    Then the critical exponent of $\Gamma$ is equal to the Hausdorff dimension of its limit set:
    \[
    \dim_{\mathcal H}(\Lambda_\Gamma) = \delta_\Gamma.
    \]
\end{theorem}
Since then, this result has inspired many generalizations. 
For a convex-cocompact action of a discrete group $\Gamma$ on a Gromov hyperbolic space, Coornaert showed in \cite{coornaert1993mesures} that Sullivan's result still holds, when the Hausdorff dimension is computed with respect to Gromov's quasi-distance on the boundary of the space.

Another direction of generalization, which we also follow is to consider actions of Anosov representations of Gromov hyperbolic groups into higher rank Lie groups.
These were introduced by Labourie in \cite{labourie_anosov_2006} and are widely regarded as the natural generalization of convex-cocompact representations to higher rank.
They stand out because of the way they reflect the dynamics and the geometry of the represented hyperbolic group onto the target Lie group and its associated flag variety.
More concretely, if $\rho: \Gamma \to \SL(d, \mathbb R)$ is a projective Anosov representation, then it admits a continuous $\rho$-equivariant embedding $(\xi^1, \xi^{d-1}): \partial_\infty \Gamma \to \mathbb P(\mathbb R^d) \times \mathbb P((\mathbb R^d)^*)$ of the Gromov boundary $\partial_\infty \Gamma$ of $\Gamma$ (see \cref{sec:anosov_definitions} for more details).
In this context, the role of the limit set is usually played by the image $\xi^1(\partial_\infty \Gamma)$ of the boundary map, which consists of the attracting points of the action of $\rho(\Gamma)$ on $\mathbb P(\mathbb R^d)$. A notable exception to this is \cite{glorieux_hausdorff_2023}, where the authors study the dimension of the symmetric limit set $(\xi^1, \xi^{d-1})(\partial_\infty \Gamma)$ instead.

In \cite{dey2022patterson}, Dey and Kapovich define a Finsler metric on the symmetric space associated with $\SL(d, \mathbb R)$, and show that the critical exponent of $\rho(\Gamma)$ with respect to this metric is equal to the Hausdorff dimension of the limit set $\xi^1(\partial_\infty \Gamma)$, computed with respect to a Gromov quasi-distance.\footnote{In fact their result holds for Anosov representations into semisimple Lie groups, but we only state it for projective Anosov representations into $\SL(d, \mathbb R)$ for ease of comparison with other results presented here.}

Sullivan's result has even seen a generalization to the pseudo-Riemannian setting by Glorieux and Monclair in \cite{glorieux2021critical}.
There, the authors define pseudo-Riemannian analogs of the critical exponent and the Hausdorff dimension and prove that they are equal for certain convex-cocompact actions on pseudo-Riemannian hyperbolic spaces (defined in \cref{sec:pseudo_riemannian_hyperbolic_spaces}).
Unlike \cite{dey2022patterson,glorieux2021critical}, here we follow \cite{glorieux_hausdorff_2023} and study the Hausdorff dimension with respect to the natural smooth structure of the flag variety where the limit set lies.
The same does Zhufeng Yao in \cite{yao2026hausdorff}, where he proves that projective Anosov subgroups of $\mathrm{PGL}(d, \mathbb R)$ with full Hausdorff dimension are uniform lattices in
$\mathrm{PGL}(2,\mathbb R)$ and that the Hausdorff dimension of the limit set equals the critical exponent of the first simple root for \emph{projective non-folding Anosov subgroups}, which are defined in the same paper.

\subsection{Generalizations involving the Falconer functional}
Motivated by the interpretation of the Cartan projection $\mu: \SL(d, \mathbb R) \to \mathfrak a^+$ as a vector-valued distance function on the symmetric space of $\SL(d, \mathbb R)$, one may give a new definition of the critical exponent as the convergence radius of a Poincaré series.
Using the Falconer functional $F_s: \mathfrak a^+ \to \mathbb R$, which was introduced by Ledrappier and Lessa in \cite{ledrappier_dimension_2023} and whose definition is recalled in \cref{sec:falconer_functional}, we define the critical exponent of a projective Anosov representation $\rho: \Gamma \to \SL(d, \mathbb R)$ as
\[
\delta_F(\rho) = \inf \left\{ s > 0 : \sum_{\gamma \in \Gamma} e^{-F_s(\mu(\rho(\gamma)))} < \infty \right\}.
\]
While not evident at first sight, this definition carries a geometric meaning, as the Poincaré series above is, up to a constant, equal to the Hausdorff measure of a covering for the limit set by ellipsoids.
In this way, the fact that the Hausdorff dimension of the limit set is less than or equal to the critical exponent, follows from the dynamical properties that the representation inherits from the assumption that it is Anosov (see \cite[Section 3]{pozzetti_anosov_2023} or \cite{ledrappier_dimension_2023} for a proof).

This suggests the following possible generalization of \cref{thm:sullivan}, appearing also in \cite{ledrappier_dimension_2023}:
\begin{vagueconjecture}
Let $\rho: \Gamma \to \SL(d, \mathbb R)$ be an irreducible projective Anosov representation of a discrete group $\Gamma$.
Under suitable assumptions, the Hausdorff dimension of the limit set is equal to the critical exponent of the Falconer functional:
\[
\dim_{\mathcal H}(\xi^1(\partial_\infty \Gamma)) = \delta_F (\rho).
\] 
\end{vagueconjecture}
The hypothesis of irreducibility is necessary, as otherwise the representation will factor through a lower-dimensional one, placing the limit set in a proper projective subspace of $\mathbb P(\mathbb R^d)$.
This will keep the Hausdorff dimension low, but may not be detected by the Falconer functional.
For example, let $\Gamma$ be a uniform lattice of $\SL(2,\mathbb R)$ and $\rho: \Gamma \to \SL(3,\mathbb R)$ given by the composition of the inclusion $\Gamma \leq \SL(2,\mathbb R)$ with the inclusion $\SL(2,\mathbb R) \to \SL(3,\mathbb R)$ as the upper-left block.
Then $\rho$ is projective Anosov and its limit set is a copy of $\partial_\infty \mathbb H^2$ inside $\mathbb P(\mathbb R^3)$, so the (Hausdorff) dimension is $1$.
On the other hand, \cref{thm:sullivan} implies that the critical exponent of the Falconer functional is equal to $3/2$.

While Pozzetti, Sambarino and Wienhard show in \cite{pozzetti_anosov_2023} that $\dim_{\mathcal H}(\xi^1(\partial_\infty \Gamma)) \leq \delta_F (\rho)$ holds for every projective Anosov representation, the reverse inequality has only been shown in some particular cases.
Namely under no additional hypotheses, for $d = 3$ by Li, Pan, and Xu in \cite{li2023dimension}, for hyperconvex representations into $\SL(d, \mathbb R)$ in \cite{pozzetti2021conformality}, and in \cite[Theorem A]{pozzetti_anosov_2023} for Zariski-dense representations into $\mathrm{SL}(d,\mathbb R)$ with Lipschitz limit set.

\subsection{Main results}
Our study takes place in the setting of projective Anosov representations with Lipschitz limit set, whose main source of examples are $\mathbf H^{p,q}$-convex-cocompact representations of discrete groups with maximal cohomological dimension (studied in \cite{beyrer2023} and recalled in \cref{sec:convex_cocompact_representations}) and $\Theta$-positive representations (see \cite{guichard2016positivity} and \cite{pozzetti_anosov_2023}).

Our first result is to show that for certain irreducible and Zariski-dense  projective Anosov representations with Lipschitz limit set, the critical exponent of the Falconer functional is equal to the Hausdorff dimension of the limit set, thus confirming the conjecture in this setting.
These assumptions are not automatically satisfied by $\Theta$-positive representations or representations of discrete subgroups with maximal cohomological dimension, but they are fulfilled by two subclasses of the latter: $\mathrm{AdS}$-quasi-Fuchsian representations (see \cref{sec:convex_cocompact_representations} for the definition) and certain deformations of fundamental groups constructed in \cite{beyrer2023}.
\begin{theorem}\label{thm:main}
    Let $\Gamma$ be a discrete group and $\rho: \Gamma \to \SL(d,\mathbb R)$ be an irreducible representation.
    If either of the following assumptions hold:
    \begin{enumerate}
        \item $\rho$ is projective Anosov, $\rho(\Gamma)$ is Zariski-dense in $\SL(d, \mathbb R)$, and the limit set $\xi^1(\partial_\infty \Gamma)$ is a Lipschitz sphere,\label{item:main_desnse_sl}
        \item $\rho$ is projective Anosov, $\rho(\Gamma)$ is Zariski-dense in $\SO(p,q)$ with $p \neq q$, and the limit set $\xi^1(\partial_\infty \Gamma)$ is a Lipschitz sphere,\label{item:main_desnse_so}
        \item $\rho(\Gamma) \subseteq \SO(d-2,2)$ is AdS-quasi-Fuchsian,\label{item:main_ads_quasi_fuchsian}
    \end{enumerate}
    then the dimension of the limit set is equal to the critical exponent of the Falconer functional.
\end{theorem}
The first case of the theorem is proven in \cite[Theorem A]{pozzetti_anosov_2023}.
There, the authors state it for strongly irreducible representations, but the proof uses two incorrect lemmas (\cite[Lemma 6.8]{pozzetti_anosov_2023} and \cite[Proposition 10.3]{labourie_anosov_2006}) that we disprove in Sections \ref{sec:equivariant_spacelike_embeddings} and \ref{sec:so_p_p_action}.
However, the original proof goes through in the Zariski-dense case, and is presented in \cref{sec:density_in_sl} nonetheless.
The other two cases of \cref{thm:main} are new and follow from the same strategy, combined with a study of the action of $\SO(p,q)$ on the space of maximal isotropic subspaces of $\mathbb R^{p,q}$.

In the case where one only assumes that the representation is strongly irreducible, certain issues may arise in the way that the tangent spaces of the limit set intersect each other.
Our next contribution is to provide examples of representations that exhibit such pathological behaviour.
Namely, that there exist strongly irreducible projective Anosov representations $\rho: \Gamma \to \SO(p,q)$ with $p \neq q$, whose limit set is a Lipschitz sphere, and such that any two tangent spaces of the limit set at different points have non-trivial intersection:

In \cite[Lemma 6.8]{pozzetti_anosov_2023}, the authors consider a strongly irreducible projective Anosov representation $\rho : \Gamma \to \mathrm{PGL}(d, \mathbb R)$ whose limit set $\xi^1(\partial_\infty \Gamma)$ is homeomorphic to a sphere $\mathbb S^{d_\rho}$ and assume that there exists a measurable $\rho$-equivariant section from the limit set to the flag variety $\mathcal F_{\{ \alpha_1, \alpha_{d_\rho + 1}\}} (\mathbb R^d)$ of type $(1, d_\rho + 1)$-flags.
They then claim that for every type $(d-d_\rho - 1, d-1)$-flag, its set of transverse flags will have positive mass with respect to every $\rho(\Gamma)$-Patterson--Sullivan measure supported on the image of the section (see \cref{def:patterson-sullivan}).

While for the most part, the proof is valid, the problem arises when it uses \cite[Proposition 10.3]{labourie_anosov_2006}.
Its statement is that for an algebraic group $G \leq \SL(d, \mathbb R)$ whose identity component acts irreducibly on $\mathbb R^d$, we can always move two subspaces of complementary dimension in transverse position.
More concretely, that there are no pairs of subspaces $C \in \mathcal G_{k}(\mathbb R^d), B \in \mathcal G_{d-k}(\mathbb R^d)$ such that
\begin{equation}\label{eq:intersection_property}
        g C \cap B \neq 0 \text{ for all } g \in G.\tag{IP}
\end{equation}

In Sections \ref{sec:equivariant_spacelike_embeddings} and \ref{sec:so_p_p_action} we provide counterexamples to both lemmas.
In particular, to disprove \cite[Proposition 10.3]{labourie_anosov_2006}, we show that the action of $\SO(d,d)$ on maximal isotropic subspaces cannot always place them in transverse position, although its identity component acts irreducibly on $\mathbb R^{2d}$.
Parametrizing the maximal isotropic subspaces by $\mathrm{O}(d)$, we see that the space of maximal isotropic subspaces has two connected components.
The key detail here is that depending on the parity of $d$, maximal isotropic subspaces in the same or in different components are never transverse.
\begin{restatable*}{proposition}{labourie}\label{ex:labourie}
    Let $S \subseteq \mathcal G_p (\mathbb R^{2p})$ be the set of maximal isotropic subspaces of $\mathbb R^{p,p}$, where $p \geq 2$.
    Then
    \[
        gV \cap W \neq 0 \text{ for all } g \in \SO(p,p)
    \]
    provided that
    \[
    V, W \text{ lie in } 
    \begin{cases}
            \text{ the same component of } S, & \text{ if }p \text{ is odd},\\
            \text{ different components of } S, & \text{ if }p \text{ is even}.
        \end{cases}
    \]
\end{restatable*}
The counterexample presented for \cite[Lemma 6.8]{pozzetti_anosov_2023}, is the action of the automorphism group of a non-degenerate Hermitian form, on the space of traceless Hermitian forms.
Strong irreducibility is proven in \cref{sec:strong_irreducibility}, while the construction and the properties of the representation are described in \cref{sec:adjoint_representations}.
The case of complex Hermitian forms has already been considered in the HDR thesis \cite{monclair2025geometrie} of Daniel Monclair, whom we thank for suggesting it as a counterexample.

As described in \cref{sec:hermitian_forms}, we consider an $n$-dimensional vector space $V$ over $\mathbb F \in \{ \mathbb R, \mathbb C, \mathbb H\}$, and a fixed Hermitian form $h_{n-1,1}$ of signature $(n-1,1)$ over $V$.
After fixing a basis of $V$, we can identify the space of Hermitian forms on $V$ with the space of matrices in $\mathbb F^{n\times n}$ that are self-adjoint with respect to $h_{n-1,1}$.
This allows us to define the space of traceless Hermitian forms $\mathrm{Herm}_\mathbb F^0(V)$, consisting of forms that correspond to traceless self-adjoint matrices.
It carries an action of the group $G=\mathrm{Aut}(h_{n-1,1}) \cap \SL(V)$ of automorphisms with determinant $1$, which preserves the non-degenerate symmetric bilinear form given by the real trace of the product of the corresponding matrices.
This gives rise to a representation $\tau: G \to\SO\left(\mathrm{Herm}^0_\mathbb F(V)\right)$ that we prove to be projective Anosov when restricted to a uniform lattice.
Since its limit set is a smooth sphere inside the projective space $\mathbb P\left(\mathrm{Herm}^0_\mathbb F(V)\right)$, the tangent space at each point $\xi^1(x)$ may be identified with the quotient of a subspace $T_{\xi^1(x)} \xi^1(\partial_\infty \Gamma)$ of $\mathrm{Herm}^0_\mathbb F(V)$ by the line $\xi^1(x)$ itself.
\begin{proposition}[\cref{prop:pozzetti}]\label{ex:pozzetti}
    Let $\mathbb F \in \{ \mathbb C, \mathbb H\}$, $G = \mathrm{Aut}(h_{n-1,1}) \cap \SL(V)$, $\Gamma \leq G$ be a uniform lattice.
    Then the representation $\tau: \Gamma \to \SO\left(p', q'\right)$ corresponding to the action of $G$ on $\mathrm{Herm}^0_\mathbb F(V)$ is projective Anosov and strongly irreducible.
    Moreover, the map $\zeta: \partial_\infty \Gamma \to \mathcal F_{1, p'}(\mathrm{Herm}_\mathbb F^0(V))$ given by
    \[
    \zeta(x) = \left(\xi^1(x), T_{\xi^1(x)} \xi^1(\partial_\infty \Gamma) \right),
    \]
    is a $\Gamma$-equivariant smooth section of the projection $$\xi^1(\partial_\infty \Gamma) \times \mathcal F_{1, p'}(\mathrm{Herm}_\mathbb F^0(V)) \to \mathcal F_{1, p'}(\mathrm{Herm}_\mathbb F^0(V)),$$ with the property that for every pair of points $x, y \in \partial_\infty \Gamma$, the intersection $\zeta^{p'}(x) \cap \zeta^{p'}(y)$ is not trivial.
    In particular, for any $x \in \partial_\infty \Gamma$ the set of non-transverse flags to $\zeta(x)$ is of full measure with respect to any $\tau(\Gamma)$-quasi-invariant measure $\mu$ supported on $\zeta(\partial_\infty \Gamma)$.
\end{proposition}

Another independent point of interest about these representations is that they provide examples of $\mathbf H^{p,q}$-convex-cocompact representations of maximal virtual cohomological dimension (see \cref{sec:representations_maximal_dimension}).
These were introduced by Beyrer and Kassel in \cite{beyrer2023}, where they show that they constitute connected components of representation varieties $\mathfrak{X}(\pi_1(N), G)$ consisting entirely of discrete and faithful representations, where $N$ is a manifold of dimension greater than $2$ and $G$ a higher rank semisimple Lie group. Moreover, they are particularly relevant in our context because they are projective Anosov representations whose limit set is a Lipschitz sphere.
A more precise description of their characteristics is summarized in the following proposition:
\begin{restatable}{proposition}{CounterexampleReps}\label{prop:adjoint_representation_convex_cocompact}
    Let $V$ be a finite dimensional vector space over $\mathbb F \in \{ \mathbb R, \mathbb C, \mathbb H \}$, $\tilde h_{p,q} \in \mathrm{Herm}_\mathbb F(V)$ be a non-degenerate Hermitian form of signature $(p,q)$, and set
    \[
    p' = \dim_\mathbb R (\mathbb F) \cdot pq, \quad q' = \dim_\mathbb R (\mathbb F) \cdot \left(\frac{p(p-1)}{2} + \frac{q(q-1)}{2}\right) + p + q - 1.
    \]
    If $\Gamma \subseteq G = \mathrm{Aut}(\tilde h_{p,q}) \cap \mathrm{SL}(p+q, \mathbb F)$ is a uniform lattice, then the representation corresponding to the action of $\mathrm{Aut}(\tilde h_{p,q})$ on the space of traceless Hermitian forms
    \[
    \tau: \Gamma \hookrightarrow G \to \SO\left(p', q'\right)
    \]
    is strongly irreducible. Moreover, it admits an invariant complete connected spacelike embedded submanifold of dimension $p'$ in $\mathbf H^{p',q'-1}$, that is diffeomorphic to $G/K$ and over which $\tau(\Gamma)$ acts cocompactly.
    Finally the representation is $\mathbf H^{p',q'-1}$-convex-cocompact if and only if $\min\{p,q\} = 1$, in which case it is $\mathbf H^{p',q'-1}$-convex-cocompact of maximal virtual cohomological dimension.
\end{restatable}

\subsection{Plan of the paper}
For the reader's convenience, we have gathered the notation and conventions used throughout the paper in \cref{sec:notation}.
In \cref{sec:anosov} we recall the definitions of projective Anosov representations in $\mathrm{SL}(d,\mathbb R)$, and $\mathbb H^{p,q}$-convex-cocompact representations.
This will be necessary in the proof of \cref{thm:main} in \cref{sec:main_result} and in the construction of the representations of \cref{sec:equivariant_spacelike_embeddings}.
Finally, in \cref{sec:so_p_p_action} we offer a counterexample to \cite[Proposition 10.3]{labourie_anosov_2006} that is more direct than the one provided by the representations we construct in the preceding section.

\section{Notation and conventions}\label{sec:notation}
We collect here the main symbols and conventions used throughout the paper.
\begin{itemize}
    \item $\Gamma$: a discrete group.
    \item $\partial_\infty \Gamma$: Gromov boundary of a hyperbolic group $\Gamma$.
    \item Singular values: $\sigma_1(g)\geq\dots\geq\sigma_d(g)$ for $g\in\SL(d,\mathbb R)$; 
    \item $\mathfrak a^+$: the closed Weyl chamber, consisting of diagonal traceless matrices with non-increasing entries in the case of $G = \SL(d,\mathbb R)$.
    \item Cartan projection $\mu: G \to \mathfrak a^+$.
    \item Limit map / limit set: For a projective Anosov representation $\rho$, the boundary map is $\xi=(\xi^1,\xi^{d-1}):\partial\Gamma\to\mathcal F_{1,d-1}(\mathbb R^d)$ and the limit set is $\Lambda_\rho=\xi^1(\partial\Gamma)\subseteq\mathbb P(\mathbb R^d)$.
    \item $\mathbf H^{p,q}$: pseudo-Riemannian hyperbolic space (negative lines in $\mathbb P(\mathbb R^{p,q+1})$); $\hat{\mathbf H}^{p,q}$ denotes its two-fold cover of unit negative vectors.
    \item $\mathbb H^n$ and $\mathbf H^n_\mathbb R$: real hyperbolic space of dimension $n$.
    \item $\mathbf H^n_\mathbb C, \mathbf H^n_\mathbb H$: complex and quaternionic hyperbolic space of dimension $n$.
    \item Grassmannian / flag varieties: $\mathcal G_k(\mathbb R^d)$ the $k$-Grassmannian; $\mathcal F_\Theta$ the partial flag variety associated to a subset $\Theta$ of simple roots $\Pi$.
    \item Roots and weights: Simple roots $\alpha_i(a)=a_i-a_{i+1}$; $\alpha_{i,j}(a)=a_i-a_j$; fundamental weights $\omega_k(a)=a_1+\dots+a_k$ for $a\in\mathfrak a$.
    \item Functionals: Falconer functional $F_s$ and unstable Jacobian $J_p^u$ given in \cref{def:functionals}. Their critical exponents are denoted by $\delta_F(\rho)$, $\delta_{J_p^u}(\rho)$.
    \item Dimension notation: $\dim_{\mathcal H}$ stands for Hausdorff dimension (ambient metric indicated by context); $\mathrm{vcd}(\Gamma)$ is the virtual cohomological dimension.
    \item “Lipschitz sphere”: An embedded topological sphere of Lipschitz regularity.
    \item Singular subspaces: For $g \in \SL(d, \mathbb R)$, $U_p(g)$ is the span of the first $p$ singular vectors (well-defined when $\sigma_p(g)>\sigma_{p+1}(g)$). 
    \item Adjoint representation: $\Ad$ is the adjoint representation and $\ad$ its differential.
    \item Conjugate transpose of a matrix (see \cref{sec:hermitian_forms}): For $A = (a_{ij})_{i,j} \in \{\mathbb R^{n\times n}, \mathbb C^{n\times n}, \mathbb H^{n\times n}\}$:
    \[
    A^* = \begin{cases}
        (\bar a_{ji})_{i,j}, &\text{ on } \mathbb C^{n\times n}, \mathbb H^{n\times n},\\
        (a_{ji})_{i,j}, &\text{ on } \mathbb R^{n\times n}
    \end{cases}
    \]
    \item Hermitian forms (see \cref{sec:hermitian_forms}):
\[
\begin{array}{c}
h_{n+1}(x, y) = \sum_{i=1}^{n+1} \bar x_i y_i \\[6pt]
h_{n,1}(x,y) = \sum_{i=1}^n \bar x_i y_i -\bar x_{n+1} y_{n+1}\\[6pt]
\tilde h_{n,1}(x,y) = \bar x_1 y_{n+1} + \bar x_{n+1} y_1 + \sum_{i=2}^n \bar x_i y_i
\end{array}
\]
for
\begin{align*}
    x &= \sum\limits_i x_i e_i, y = \sum\limits_i y_i e_i , \text{ if } \mathbb F \in \{\mathbb R, \mathbb C\},\\
    x &= \sum\limits_i e_i x_i, y = \sum\limits_i e_i y_i, \text{ if } \mathbb F = \mathbb H.
\end{align*}
    \item Hermitian matrices (see \cref{sec:hermitian_forms}):
\[
H_n = I_{n}, \quad
H_{n-1,1} = 
\begin{pmatrix}
I_n & 0 \\
0 & -1
\end{pmatrix}, \quad
\tilde H_{n-1,1} = 
\begin{pmatrix}
0 & 0 & 1 \\
0 & I_{n-1} & 0 \\
1 & 0 & 0
\end{pmatrix}.
\]
    
\end{itemize}

\section{Anosov representations}\label{sec:anosov}
One of the main choices when generalizing Sullivan's result to higher rank Lie groups is to find a suitable replacement for the limit set.
In the context of Anosov representations, a natural candidate for this is the embedding of the Gromov boundary of $\Gamma$ into the projective space, given by the boundary map of the representation.
The parts of this theory that we will need are presented in \cref{sec:anosov_definitions}, using the Lie theory of $\SL(d,\mathbb R)$ recalled in \cref{sec:lie_theory}. 

A particular class of Anosov representations that we will be interested in are $\mathbf H^{p,q}$-convex-cocompact representations, defined in \cref{sec:convex_cocompact_representations}.
These include $\mathbf H^{p,q}$-convex-cocompact representations of maximal cohomological dimension, which are one of the main sources of examples of Anosov representations with Lipschitz limit set.
Here they will appear in the third case of \cref{thm:main}, and in the examples of representations that we construct in \cref{sec:adjoint_representations}.

\subsection[\texorpdfstring{Lie theory for $\mathrm{SL}(d,\mathbb R)$}{Lie theory for SL(d,R)}]{Lie theory for $\SL(d,\mathbb R)$}\label{sec:lie_theory}
In order to define Anosov representations, we will need certain facts about the structure of $\SL(d,\mathbb R)$ and its Lie algebra $\ssl(d,\mathbb R)$.
For this reason, in this section we include a brief overview of the relevant Lie theory.

The \emph{Cartan subalgebra} $\mathfrak a$ of $\ssl(d,\mathbb R)$ is the subalgebra of diagonal traceless matrices, and we will fix the Weyl chamber $\mathfrak a^+$ of diagonal traceless matrices with non-increasing entries.
The Cartan projection will be denoted with $\mu: \SL(d,\mathbb R) \to \mathfrak a^+$.
The following functionals defined over $\mathfrak a$ will be used:
\begin{align*}
    \alpha_{i,j}(a) &= a_i - a_j, \quad 1 \leq i \neq j \leq d,\\
    \alpha_k(a) &= a_k - a_{k+1}, \quad 1 \leq k \leq d-1,\\
    \omega_k(a) &= a_1 + \cdots + a_k,
\end{align*}
for $a = \diag(a_1, \cdots, a_k) \in \mathfrak a$.

Recall that the simple positive roots for this Weyl chamber are given by $\Pi = \{ \alpha_1, \cdots, \alpha_{d-1} \}$.
For ease of notation, we will often identify subsets of $\Pi$ with subsets of $\{1, \cdots, d-1\}$ according to the index of the roots.
To each such subset $\Theta \subseteq \Pi$ we associate the subset $\iota \Theta = \{ d-\theta : \theta \in \Theta \} \subseteq \Pi$, and the Levi subspace $\mathfrak a_\Theta$, which is the intersection of the walls of the Cartan subalgebra that correspond to the roots in $\Theta^c$:
\[
\mathfrak a_\Theta = \bigcap_{i \not \in \Theta} \ker \alpha_i.
\]
The flag variety $\mathcal F_\Theta$ that corresponds to $\Theta$ is the flag space of subspaces with the respective dimensions:
\[
\mathcal F_\Theta = \left\{ V_{\theta_1} \leq \cdots < V_{\theta_k} : \dim V_{\theta_i} = \theta_i \text{ for } \theta_i \in \Theta \right\}.
\]

\subsection{Definition and boundary maps of Anosov representations}\label{sec:anosov_definitions}
Here we recall the definition and properties of Anosov representations and their boundary maps.
They generalize convex-cocompact representations to higher-rank Lie groups, and thus provide a natural setting for extending Sullivan's result.
The following definition is in fact a characterization proved in \cite{kapovich2017anosov,bochi2019anosov} for the original one of \cite{labourie_anosov_2006}.
\begin{definition}[Anosov representation]\label{def:anosov}
    Let $\rho: \Gamma \to \SL(d, \mathbb R)$ be a linear representation of a discrete group $\Gamma$ and $p \in \llbracket 1, d-1 \rrbracket$.
    We say that $\rho$ is \emph{$p$-Anosov} if there exist constants $c, C>0$ such that for all $\gamma \in \Gamma$:
    \[
        \frac{\sigma_{p+1}}{\sigma_p}(\rho(\gamma)) \leq C e^{-c |\gamma|},
    \]
    where $\sigma_1(g) \geq \cdots \geq \sigma_d(g)$ are the singular values of $g$. When $p=1$, we say that the representation is projective Anosov.
\end{definition}
\subsubsection{Gromov hyperbolic groups}
One of the main features of Anosov representations is that they reflect the hyperbolic geometry of $\Gamma$, as a Gromov hyperbolic group.
\begin{definition}[Gromov hyperbolic group]\label{def:hyperbolic_group}
    Consider a finitely generated group $\Gamma$ equipped with the word-length metric $d$.
    We say that $\Gamma$ is $\delta$-hyperbolic if for all $x,y,z \in \Gamma$ we have
    \[
    (x,y) \geq \min \{ (x,z), (y,z) \} - \delta,
    \]
    where the Gromov product $(x,y)$ is defined as
    \[
    (x,y) = \frac{1}{2} \left(d(x,e) + d(y,e) - d(x,y)\right).
    \]
    In that case, the Gromov boundary $\partial_\infty \Gamma$ can be defined as the set of equivalence classes of sequences in $\Gamma$ that escape to infinity under the equivalence relation
    \[
    (\gamma_n) \sim (\gamma_n') \text{ if and only if } \lim_{n \to \infty} (\gamma_n, \gamma_n') = \infty.
    \]
\end{definition}
An important fact of hyperbolic group theory is that for a hyperbolic group $\Gamma$, the boundary $\partial_\infty \Gamma$ admits a topology for which $\bar \Gamma = \Gamma \cup \partial_\infty \Gamma$ is compact.
Here we will not need the explicit construction of this topology, so we refer the reader to \cite{ghys2013groupes} for more details.

\subsubsection{Boundary maps of Anosov representations}
The following theorem summarizes some key properties of Anosov representations that exhibit their dynamical behaviour and how they are related to hyperbolic groups.
Essentially, it states that a group $\Gamma$ admitting such a representation is Gromov hyperbolic and admits a limit map which embeds the Gromov boundary $\partial_\infty \Gamma$ into the flag variety of $\mathrm{SL}(d,\mathbb R)$ in a dynamics-preserving manner.

In the expression of limit maps that appear in the theorem below, we use the notation 
\[
    U_p(g) = k_g \cdot  (\mathbb R e_1 \oplus \cdots \oplus \mathbb R e_p)
\]
for an element $g \in \SL(d,\mathbb R)$ that has a gap $\sigma_p(g) > \sigma_{p+1}(g)$ in its singular values and with Cartan decomposition $g = k_g e^{\mu(g)} l_g$.
It should be noted that for a general element $g \in \SL(d,\mathbb R)$, the subspace $U_p(g)$ is not well-defined, since the compact factors $k_g, l_g$ in the Cartan decomposition are not unique.
It is however well-defined when $\sigma_p(g) > \sigma_{p+1}(g)$, since then the first $p$-columns of $k_g$ are uniquely determined for any Cartan decomposition of $g$.
Indeed, the set $\{ k_g e_i : 1 \leq i \leq d \}$ is an orthogonal basis of $\mathbb R^d$ that consists of the axes of the ellipsoid $\{g v : v \in \mathbb R^d, \|v\| = 1\}$.
Because the lengths of the axes are given by the singular values of $g$, the $p$ longest ones are uniquely determined.
\begin{theorem}[\cite{labourie_anosov_2006,kapovich2017anosov}]\label{thm:boundary_map}
    \index{Anosov!boundary maps}
    Let $\rho: \Gamma \to \SL(d, \mathbb R)$ be a $p$-Anosov representation of a discrete group $\Gamma$.
    Then the following hold:
    \begin{enumerate}[label=(\roman*)]
        \item $\Gamma$ is Gromov hyperbolic
        \item there exists a $\rho$-equivariant continuous map $\xi = (\xi^p, \xi^{d-p}): \partial_\infty \Gamma \to \mathcal F_{p, d-p}(\mathbb R^d) $ over its Gromov boundary $\partial_\infty \Gamma$ such that for all $x \in \partial_\infty \Gamma$ and any quasi-geodesic 
        $\gamma_n$ that converges to $x$ we have
        \begin{align*}
            \xi^p(x) &= \lim_{n \to \infty} U_p(\rho(\gamma_n))\\
            \xi^{d-p}(x) &= \lim_{n \to \infty} U_{d-p}(\rho(\gamma_n)),
        \end{align*}
        \item The boundary map $\xi$ is dynamics-preserving, i.e.\ 
        \[
        \rho(\gamma_n) l \to \xi^p(x)
        \]
        for any sequence $(\gamma_n)_n$ of $\Gamma$ for which $\gamma_n \to x$ and $\gamma_n^{-1} \to y$ as $n \to \infty$, and any subspace $l \in \mathcal G_{p}(\mathbb R^d)$ that is transverse to $\xi^{d-p}(y)$.
        \item The boundary maps $(\xi^p, \xi^{d-p})$ are transverse, in the sense that for every $x, y \in \partial_\infty \Gamma, \xi^p(x) \oplus \xi^{d-p}(y) = \mathbb R^d$ unless $x=y$.
    \end{enumerate}
\end{theorem}
Given a projective Anosov representation $\rho: \Gamma \to \SL(d,\mathbb R)$, we will call $\xi_\rho^1(\partial_\infty \Gamma) \subseteq \mathbb P(\mathbb R^d)$ the \emph{limit set} of the representation.
The terminology is justified by the fact that in the case of a rank-one representation into $\SO_0(n,1)$, the limit set is the usual limit set of the group action on the hyperbolic space, i.e.\ the set of accumulation points of an orbit in the boundary of the space.
Indeed, we may consider the projective model of the hyperbolic space
\[
\mathbf H_\mathbb R^d = \{ \mathbb Rx \in \mathbb P(\mathbb R^{d+1}) : \langle x,x \rangle_{d,1} < 0 \},
\]
where $\langle \cdot, \cdot \rangle_{d,1}$ denotes a non-degenerate symmetric bilinear form of signature $(d,1)$ on $\mathbb R^{d+1}$.
In that model, the boundary $\partial \mathbf H_\mathbb R^d$ is the set of isotropic lines in $\mathbb P(\mathbb R^{d+1})$.
Given now a convex-cocompact representation $\rho:\Gamma \to \SO_0(d,1)$, the Milnor-Svarc lemma implies that the orbit map $\Gamma \to \mathbb H^d$ is a quasi-isometry.
Using a standard result in the theory of hyperbolic spaces, one may extend this to a continuous map $\bar \Gamma \to \bar{\mathbb H^d} = \mathbb H^d \cup \partial \mathbb H^d$.
Furthermore, its restriction to $\partial_\infty \Gamma$ gives rise to a topological embedding $\partial_\infty \Gamma \hookrightarrow \partial \mathbb H^d$ that coincides with the limit map $\xi_\rho^1: \partial_\infty \Gamma \to \mathbb P(\mathbb R^{d+1})$ of \cref{thm:boundary_map}.
Unravelling the definitions, we see that
\[
\Lambda_{\rho(\Gamma)} = \left(\rho(\Gamma) \cdot p\right)' = \overline{\rho(\Gamma) \cdot p} - \rho(\Gamma) \cdot p = \xi_\rho^1(\partial_\infty \Gamma)
\]
where $p$ is any fixed point in $\mathbb H^d$.

\subsection[\texorpdfstring{Convex-cocompact representations in $\mathrm{SO}(p,q+1)$}{Convex-cocompact representations in SO(p,q+1)}]{Convex-cocompact representations in $\mathrm{SO}(p,q+1)$}\label{sec:convex_cocompact_representations}
Recall that the pseudo-Riemannian hyperbolic space $\mathbf H^{p,q}$ is defined as the set of negative lines in $\mathbb P(\mathbb R^{p+q+1})$, with respect to a non-degenerate symmetric bilinear form of signature $(p,q+1)$:
\[
\mathbf H^{p,q} = \{ \mathbb R x \in \mathbb P(\mathbb R^{p+q+1}) : \langle x,x \rangle_{p,q+1} < 0 \}.
\]
It is a pseudo-Riemannian manifold of signature $(p,q)$, and will be revisited in \cref{sec:equivariant_spacelike_embeddings}.
Here we shall define a class of representations into its isometry group, namely $\mathbf H^{p,q}$-convex-cocompact representations, focusing on the case of ones with maximal cohomological dimension.
The reason to consider the latter is twofold; namely they provide a rich source of examples of Anosov representations with Lipschitz limit set, and they are the union of connected components of representations of groups that are not necessarily fundamental groups of surfaces (see Theorem 1.3 in \cite{beyrer2023}).
\subsubsection[\texorpdfstring{$\mathbf H^{p,q}$- and $\mathrm{AdS^p}$-convex-cocompact representations}{Hp,q- and AdSp-convex-cocompact representations}]{$\mathbf H^{p,q}$- and $\mathrm{AdS^p}$-convex-cocompact representations}
Convex-cocompact subgroups of $\SO(n,1)$ are a well-studied class of discrete subgroups with rich geometric and dynamical properties.
Here they act as a starting point, since they are the class of subgroups considered by Sullivan in \cref{thm:sullivan}.
In \cite{mess2007lorentz}, \cite{merigot2012anosov} and \cite{barbot2015deformations}, Mess and later Barbot and Mérigot introduced a generalization of convex-cocompactness to the pseudo-Riemannian setting of Anti-de Sitter geometry.
With $\mathbf H^{p,q}$-convex-cocompact representations, Danciger, Guéritaud and Kassel extended this notion in \cite{danciger2018convex} to the more general setting of pseudo-Riemannian hyperbolic spaces $\mathbf H^{p,q}$.
\begin{definition}[Convex-cocompact representations]\label{def:convex_cocompact_representations}
    Let $\rho: \Gamma \to \SO(p,q+1)$ be a representation of a discrete group $\Gamma$ with finite kernel and discrete image.
    \begin{enumerate}[label=(\roman*)]
        \item $\rho$ is $\mathbf H^{p,q}$-convex-cocompact if $\rho(\Gamma)$ acts properly discontinuously and cocompactly on some closed properly convex (in the sense of $\mathbb P(\mathbb R^{p+q+1})$)subset 
        of $\mathbf H^{p,q}$ with nonempty interior.
        \item $\rho$ is $\mathrm{AdS}^p$-convex-cocompact if it is $\mathbf H^{p,1}$-convex-cocompact.
    \end{enumerate}
\end{definition}

For the proof of the following fact, we refer to Lemma 6.3 of \cite{danciger2017convex}. It will be used in \cref{prop:adjoint_representation_convex_cocompact} to show that the representations constructed there are $\mathbf H^{p,q}$-convex-cocompact only for forms of signature $(p,1)$.
\begin{fact}\label{fact:convex_cocompact_hyperbolic}
    If a group $\Gamma$ admits an irreducible $\mathbf H^{p,q}$-convex-cocompact representation, then $\Gamma$ is Gromov hyperbolic.
\end{fact}
\subsubsection[\texorpdfstring{$\mathbf H^{p,q}$-convex-cocompact representations of maximal virtual cohomological dimension}{Hp,q-convex-cocompact representations of maximal virtual cohomological dimension}]{$\mathbf H^{p,q}$-convex-cocompact representations of maximal virtual cohomological dimension}\label{sec:representations_maximal_dimension}
The next step will be to consider which of the above cases admit Lipschitz limit sets.
In \cite[Lemma 11.9.(1)]{danciger2017convex}, it is shown that the limit set of an $\mathbf H^{p,q}$-convex-cocompact representation is homeomorphic to a closed subset of the sphere $\mathbb S^{p-1}$.
Of special importance is the case where the limit set is the whole sphere, or equivalently (as shown in \cite[Corollary 11.10.(1)]{danciger2017convex}) when the group attains maximal virtual cohomological dimension.
\begin{definition}[Representations of maximal virtual cohomological dimension]\label{def:representations_maximal_dimension}
    An $\mathbf H^{p,q}$-convex-cocompact $\rho:\Gamma \to \SO(p,q+1)$ is said to be  of \emph{maximal virtual cohomological dimension} if either of the three equivalent conditions holds:
    \begin{enumerate}[label=(\roman*)]
        \item the virtual cohomological dimension $\mathrm{vcd}(\Gamma)$ of $\Gamma$ is equal to $p$,
        \item the limit set $\xi^1_\rho(\partial_\infty \Gamma)$ is homeomorphic to the sphere $\mathbb S^{p-1}$,
        \item the boundary $\partial_\infty \Gamma$ is homeomorphic to the sphere $\mathbb S^{p-1}$.
    \end{enumerate}
    If $\rho$ is $\mathbb H^{p-1,1}$-convex-cocompact of maximal virtual cohomological dimension, we shall follow \cite{monclair2023gromov} and call it \emph{$\mathrm{AdS}^p$-quasi-Fuchsian} (or \emph{$\mathrm{AdS}$-quasi-Fuchsian} when the dimension is clear from the context).
    If moreover $\rho$ preserves a totally geodesic copy of $\mathbb H^{p-1}$ in $\mathbb H^{p-1,1}$, we say that it is \emph{$\mathrm{AdS}^p$-Fuchsian}.
\end{definition}
For completeness purposes, we recall that the cohomological dimension of a group $\Gamma$ is the largest integer $n \in \mathbb N$ for which there exists a $\mathbb Z[\Gamma]$-module $M$ such that $H^n(\Gamma, M) \neq 0$.
If $\Gamma$ is virtually torsion-free (as is the case when $\Gamma$ is a finitely generated subgroup of a linear group, by Selberg's lemma), then all finite-index subgroups of $\Gamma$ have the same cohomological dimension.
We define the virtual cohomological dimension $\mathrm{vcd}(\Gamma)$ of $\Gamma$ to be this common cohomological dimension of any of its torsion-free finite-index subgroups.

A geometric characterization of $\mathbf H^{p,q}$-convex-cocompact representations of maximal virtual cohomological dimension is given in \cite[Corollary 1.14]{beyrer2023}, where it is shown that they are precisely the representations of Gromov hyperbolic groups with virtual cohomological dimension $p$ that act properly discontinuously and cocompactly on a $p$-dimensional spacelike graph of a Lipschitz map in the warped product coordinates (such as the ones given in \cref{lem:warped_product_structure}) of $\hat{\mathbf H}^{p,q}$.
One of our contributions is to provide examples of such representations, for which the spacelike graph is an embedding of a symmetric space (see \cref{prop:adjoint_representation_convex_cocompact}).

Along with $\Theta$-positive representations, representations of maximal virtual cohomological dimension constitute one of the main classes of Anosov representations whose limit set is Lipschitz.
One could argue that they are an especially rich source of examples since they form a union of connected components of the representation variety.
More concretely, the following is a direct consequence of Corollary 1.14 and Proposition 3.19 of \cite{beyrer2023}:
\begin{corollary}[Lipschitz limit set for representations of maximal virtual cohomological dimension]\label{cor:lipschitz_limit}
    The limit set of an $\mathbf H^{p,q}$-convex-cocompact representation of maximal virtual cohomological dimension is a Lipschitz submanifold of $\partial \mathbf H^{p,q}$, homeomorphic to~$\mathbb S^{p-1}$.
\end{corollary}

\section{Hausdorff dimension equals critical exponent of the Falconer functional}\label{sec:main_result}
The goal of this section is to prove \cref{thm:main}, which states that under certain conditions, the Hausdorff dimension of the limit set of a projective Anosov representation is equal to the critical exponent of the Falconer functional.
The definitions of the functionals needed, along with some of their properties, are given in \cref{sec:falconer_functional}.
The strategy of the proof of the analogous result in \cite{pozzetti_anosov_2023} is sketched in \cref{sec:proof_strategy}.
In \cref{sec:proof_of_theorem} we provide the missing details for each of the three cases of the theorem.
\subsection{Falconer functional and unstable Jacobian}\label{sec:falconer_functional}
Ledrappier and Lessa introduced in \cite{ledrappier_dimension_2023} the Falconer functional, whose critical exponent provides a dynamical invariant that is (under certain assumptions, see \cref{thm:main}) equal to the Hausdorff dimension of the limit set of the representation.
With \cref{lem:functional_relations} we shall relate it to the unstable Jacobian, introduced in \cite[Theorem A]{pozzetti_anosov_2023}, in order to provide a proof for the lower bound of the Hausdorff dimension of the limit set.
\begin{definition}\index{Falconer!functional}\index{unstable Jacobian}\label{def:functionals}
    For $s \geq 0$ we define the \emph{Falconer functional} $F_s: \mathfrak a^+ \to \mathbb R_{\geq 0}$ as
    \[
        F_s(a) = \max_{0 \leq p \leq d-2} \left\{ \alpha_{1,2}(a) + \cdots + \alpha_{1,p+1}(a) + (s-p)\alpha_{1,p+2}(a) \right\}
    \]
    and for $p \in \llbracket 1, d-1 \rrbracket$ we define the \emph{unstable Jacobian} $J_p^u: \mathfrak a^+ \to \mathbb R$ as:
    \begin{align*}
    J_p^u &= (p+1)\omega_1 - \omega_{p+1} =
    \alpha_{1,2} + \cdots + \alpha_{1,p+1}.
    \end{align*}
\end{definition}

The critical exponent of a group is generalized into an invariant associated to a family of functionals, which captures the dynamical behaviour of the representation.
We shall refer to a functional $\phi: \mathfrak a^+ \to \mathbb R$ as \emph{positive} if it is positive in the interior of the Weyl chamber $\mathfrak a^+$.
\begin{definition}
    Let $(\phi_s: \mathfrak a^+ \to \mathbb R)_s$ be an increasing family of positive functionals indexed by $s > 0$ and $\rho: \Gamma \to \SL(d,\mathbb R)$ be a representation of a finitely generated discrete group $\Gamma$.
    We define the critical exponent $\delta_\phi(\rho)$ of $(\phi_s)_s$ with respect to $\rho$ as
    \[
    \delta_\phi(\rho) = \inf \left\{ s>0: \sum_{\gamma \in \Gamma} e^{-\phi_s(\mu(\rho(\gamma)))} < \infty \right\}.
    \]
\end{definition}
Straight from the definitions, we see that $F_s(a) = s \alpha_1(a)$ for a representation in $\SO(d-1,1)$, hence the critical exponent $\delta_F(\rho)$ of the Falconer functional coincides with the critical exponent $\delta_{\rho(\Gamma)}
$ of the group $\rho(\Gamma)$.
In light of this observation, one can see \cref{thm:main} as an analog of the result in \cite{sullivan1979density} to higher rank.

We end this section by comparing the critical exponents of the unstable Jacobian and the Falconer functional, which will be used when we describe the strategy of the proof in \cref{sec:proof_strategy}.
\begin{lemma}\label{lem:functional_relations}
    For $p \in \llbracket 1, d-1 \rrbracket$:
\[
\max \{p, p \delta_{J_p^u}(\rho) \} \geq \delta_F(\rho).
\]
\end{lemma}
\begin{proof}
It will be convenient for us to introduce the following family of auxiliary  functionals $\Psi_s^p: \mathfrak a^+ \to \mathbb R$ for $p \in \llbracket 0, d-2 \rrbracket$ and $s \in \mathbb R$, given by
\[
\Psi_s^p = 
    \alpha_{1,2} + \cdots + \alpha_{1,p+1} + (s - p)\alpha_{1,p+2}
\]
which give an expression for the Falconer functional that will be easier to handle:
\[
F_s(a) = \max_{p \in \llbracket 0, d-2 \rrbracket} \Psi_s^p(a) = \Psi_s^{p_0}(a) \text{ for } s \in [p_0, p_0 + 1], \text{ for all } a \in \mathfrak a^+
\]
The equality $F_s(a) = \Psi_s^{p_0}(a)$ for $s \in [p_0, p_0 + 1]$ follows from the definitions.
The next step is to show that for all $s \geq p, p \in \llbracket 1, d-1 \rrbracket$, we have 
\[
F_s(a) \geq \Psi_s^{p-1}(a) \geq \frac{s}{p} J_p^u(a),
\]
which is obtained by:
\begin{align*}
    F_s \geq \Psi_s^{p - 1} &=
    \alpha_{1,2} + \cdots + \alpha_{1,p} + (s-(p-1)) {\alpha_{1,p+1}}\\
    &= \alpha_{1,2} + \cdots + \alpha_{1,p} + \alpha_{1,p+1} + (s-p) {\alpha_{1,p+1}}\\
    &= J_p^u + \left(\frac{s}{p}-1\right) \cdot p \ {\alpha_{1,p+1}}\\
    &\geq
    J_p^u + \left(\frac{s}{p}-1\right) J_p^u =
    \frac{s}{p} J_{p}^u.
\end{align*}
Thus, for $s > \max\{ p, p\ \delta_{J_p^u}(\rho) \}$, we have that the Poincaré series of $F_s$ converges, giving us the wanted result.
\end{proof}

\subsection{Patterson--Sullivan measures}\label{sec:ps_measures}
To obtain the original result of \cite{sullivan1979density}, Sullivan constructs a family of measures supported on the limit set, indexed by the points of the hyperbolic space, that measure the density of the limit set at infinity, as seen from each point of $\mathbb H^n$.
This construction was generalized in \cite{quint2002mesures} to construct measures on flag varieties of higher rank Lie groups, and is one of the main tools used in \cite{pozzetti_anosov_2023}.

\begin{definition}[Quint--Patterson--Sullivan measure]\index{Patterson--Sullivan measure}\label{def:patterson-sullivan}
    For a discrete subgroup $H < \PSL(d, \mathbb R)$ and a functional $\phi \in (\mathfrak a_\Theta)^*$,
    an $(H, \phi)$-Patterson--Sullivan measure is a finite Radon measure $\mu$ on $\mathcal F_\Theta$ such that for every $h \in H$
    \[
        \frac{\d h_* \mu}{\d\mu}(x) = e^{-\phi(b_\Theta(h^{-1},x))}, \text{ for all } x \in \mathcal F_\Theta(\mathbb R^d).
    \]
\end{definition}
The Busemann cocycle $b_\Theta: \PSL(d,\mathbb R) \times \mathcal F_\Theta \to \mathfrak a_\Theta$ featured in the above definition will not be used directly in our proof, so we refer the reader to \cite{pozzetti_anosov_2023} for its definition.

The following notion of irreducibility for measures on flag varieties is phrased in terms of representations in \cite{pozzetti_anosov_2023}:
\begin{definition}[Irreducible measure]\index{irreducible representation!with respect to measure}
    \label{def:mu_irreducible}
Let $\Theta \subseteq \Pi$ and $\mu$ be a measure on $\mathcal F_\Theta(\mathbb R^d)$.
We say that $\mu$ is irreducible if there is no element in $\mathcal F_{i \Theta}(\mathbb R^d)$ whose annihilator is of full measure, i.e.\ for all $y \in \mathcal F_{i \Theta}(\mathbb R^d)$:
\[
    \mu(\mathrm{Ann}(y)) < \mu(\mathcal F_{\Theta}(\mathbb R^d)),
\]
where the annihilator $\mathrm{Ann}(y)$ of $y \in \mathcal F_{i \Theta}(\mathbb R^d)$ is defined as
\[
\mathrm{Ann}(y) = \{ x \in \mathcal F_\Theta(\mathbb R^d) : x \text{ is not transverse to } y \}.
\]
\end{definition}
The reason to consider $(\rho(\Gamma), \phi)$-Patterson--Sullivan measures that are irreducible is that they admit a family $\{U_\gamma\}_{\gamma \in \Gamma}$ of open sets whose mass is proportional to $e^{-\phi(\mu(\rho(\gamma)))}$ (see \cref{lem:measure_estimation}).

\subsection[\texorpdfstring{Sketch of the proof of \cref{thm:main}}{Sketch of the proof of Theorem 2}]{Sketch of the proof of \cref{thm:main}}\label{sec:proof_strategy}
Since our main result is a modification of the one appearing in \cite{pozzetti_anosov_2023}, we begin by outlining the proof of \cite[Theorem A]{pozzetti_anosov_2023}.
We take care in pointing out the part that relies on \cite[Proposition 10.3]{labourie_anosov_2006} and \cite[Lemma 6.8]{pozzetti_anosov_2023}, which are the places around which we need to modify the proof to obtain \cref{thm:main}.

For the proof of the upper bound of \cref{thm:main}, we refer to \cite[Theorem 3.1]{pozzetti_anosov_2023}, where the authors follow a classical recipe to obtain the upper bound of the Hausdorff dimension.
More precisely, they cover the limit set with projective ellipsoids, and then each ellipsoid by balls of different radii.
The inequality then follows because the Hausdorff content of this covering is dominated by the Dirichlet series of the Falconer functional $F$.
We would like to stress that the only part of the hypothesis used in the upper bound is that the representation is projective Anosov, meaning that it still holds without requiring that $\xi^1(\partial_\infty \Gamma)$ is a Lipschitz submanifold of $\mathbb P(\mathbb R^d)$ or any of the additional assumptions of \cref{thm:main}.

denoting by $d_\rho$ the dimension of the limit set, \cref{lem:functional_relations} yields $\delta_F(\rho) \leq \max\{d_\rho, d_\rho \delta_{J_{d_\rho}^u}(\rho)\}$, so that the proof of the lower bound can be reduced to the inequality
\begin{align}\label{eq:lower_bound}
    \delta_{J^u_{d_\rho}}(\rho) \leq 1.\tag{LB}
\end{align}
For this, one employs the method of Patterson--Sullivan--Quint, which is based on finding a $(J^u_{d_\rho}, \rho(\Gamma))$-Patterson--Sullivan measure $\nu$ and a collection of open sets $\{U_\gamma\}_{\gamma \in \Gamma}$ whose mass is approximately given by $e^{-J_{d_\rho}^u(\mu(\rho(\gamma)))}$, and that do not intersect too much, in the sense that
\begin{align*}\label{eq:GoodMeasure}
    \nu(U_\gamma) \geq C_\nu e^{-J_{d_\rho}^u(\mu(\rho(\gamma)))} \text{ and } 
    M \equaldef \sup_{n \in \mathbb N}
    \max
    \left\{\sharp A : A \subseteq \Gamma_n,  \bigcap_{\gamma \in A}U_\gamma \neq \emptyset \right\} < \infty\tag{MP}
\end{align*}
holds, where $\Gamma_n = \{ \gamma \in \Gamma : |\gamma| = n \} $, and $C_\nu$ is a positive constant.
Having established the existence of such a measure, we can obtain \cref{eq:lower_bound} by first bounding uniformly in $n$:
\begin{align*}
    \sum_{\gamma \in \Gamma_n}
    e^{-J_{d_\rho}^u(\mu(\rho(\gamma)))} \leq
    C_\nu \sum_{\gamma \in \Gamma_n}
    \nu(U_\gamma) \leq M C_\nu\ \nu(\mathcal F_{\{ 1, d-1\}}(\mathbb R^d)) < \infty.
\end{align*}
Combining this with the bound implied by the Anosov property of $\rho$
\begin{align*}
J_{d_\rho}(\mu(\rho(\gamma))) &\geq \alpha_{12} (\mu(\rho(\gamma))) \geq C|\gamma| - b
\end{align*}
one concludes that (up to a multiplicative constant)
\begin{align*}
\sum_{\gamma \in \Gamma} e^{-(s+1)J^u_{d_\rho}(\mu(\rho(\gamma)))} &=
\sum_{n\geq 0} \sum_{\gamma \in \Gamma_n} e^{-s J^u_{d_\rho}(\mu(\rho(\gamma)))}
    e^{-J^u_{d_\rho}(\mu(\rho(\gamma)))}\\
    &\leq
    M C_\nu\ \nu(\mathcal F_{\{ 1, d-1\}}(\mathbb R^d)) \
\sum_{n\geq 0} e^{-s(Cn - b)} < \infty,
\end{align*}
for $s > 0$, which implies \cref{eq:lower_bound}.

While a cover satisfying the intersection property on the right hand side of \cref{eq:GoodMeasure} can be constructed using the Anosov property, the estimation of $\nu(U_\gamma)$ is more delicate.
Using the properties of a Busemann cocycle, one can show that it is implied by the irreducibility of~$\nu$:
\begin{lemma}[Lemma 5.15 of \cite{pozzetti_anosov_2023}]\label{lem:measure_estimation}
    Let $\Theta \subseteq \Pi$ and consider a functional $\phi \in (\mathfrak a_\Theta)^*$, and a representation $\rho: \Gamma \to \SL(d,\mathbb R)$.
    If $\nu^\phi$ is an irreducible $(\rho(\Gamma), \phi)$-Patterson--Sullivan measure over $\mathcal F_\Theta(\mathbb R^d)$, then there exists $\alpha_0 >0$ such that for all $\alpha \in (0, \alpha_0)$ there exists a family of open sets $\{U_{\gamma, \alpha}\}_{\gamma \in \Gamma}$ satisfying 
    \[
    c e^{-\phi(\mu(\rho(\gamma)))} \leq \nu^\phi(U_{\gamma, \alpha}) \leq C e^{-\phi(\mu(\rho(\gamma)))} \text{ for all } \gamma \in \Gamma,
    \]
    where $c, C > 0$ are constants independent of $\gamma$.

\end{lemma}
To construct a $(J_{d_\rho}, \rho(\Gamma))$-Patterson--Sullivan measure, one may proceed as in \cite[Section 2]{pozzetti_anosov_2023}, and define the section $\zeta: \xi^1(\partial_\infty \Gamma) \to \mathcal F_{\{ \alpha_1, \alpha_{d_\rho + 1}\}}(\mathbb R^d)$ of the projection $\mathcal F_{\{ \alpha_1, \alpha_{d_\rho + 1}\}}(\mathbb R^d) \to \mathbb P(\mathbb R^d)$ by 
\[
\zeta(\xi^1(x)) = (\xi^1(x), T_{\xi^1(x)} \xi^1(\partial_\infty \Gamma)),
\]
where $T_{\xi^1(x)} \xi^1(\partial_\infty \Gamma)$ is the tangent space to $\xi^1(\partial_\infty \Gamma)$ at $\xi^1(x)$.
It is almost-everywhere defined due to the Lipschitz regularity of $\xi^1(\partial_\infty \Gamma)$, and is $\rho$-equivariant.
Integrating a Lebesgue volume form over every tangent space $T_{\xi^1(x)} \xi^1(\partial_\infty \Gamma)$, one obtains a measure on $\xi^1(\partial_\infty \Gamma)$, which can be pushed forward to a measure $\nu$ on $\zeta(\partial_\infty \Gamma)$, thus obtaining a $(J_{d_\rho}, \rho(\Gamma))$-Patterson--Sullivan measure.

However, contrary to one's intuition and to the assertion of \cite[Lemma 6.8]{pozzetti_anosov_2023}, the strong irreducibility of $\rho$ and the fact that $\nu$ is supported on the image of the equivariant section $\zeta$ do not imply that $\nu$ is irreducible (see \cref{sec:adjoint_representations} for a family of counterexamples).
More specifically, in the proof of this lemma, the authors argue by contradiction that if the result is not true, then there exist subspaces $W_0 \in \mathcal G_{d_\rho + 1}$ and $V \in \mathcal G_{d-d_\rho - 1}$ such that
\[
\rho(\gamma) V \cap W_0 \neq 0
\]
for all $\gamma \in \Gamma$.
Indeed, if $\rho$ is not $\mu$-irreducible, then there exist $(W_0, P_0) \in \mathcal F_{{\alpha}_{d-d_\rho - 1}, {\alpha}_{d-1}}$ such that $\mathrm{Ann}(W_0,P_0)$ is of full $\mu$ measure.
In fact, the quasi-equivariance of $\mu$ along with the dynamics of $\rho$ imply that one can choose $W_0, P_0$ such that $P_0 \in \xi^{d-1}_\rho(\partial_\infty \Gamma)$, meaning that for $\mu$-almost every $(\xi^1_\rho(x), \zeta(x)^{d_\rho + 1}) \in \zeta(\partial_\infty \Gamma)$, we have that $\zeta(x)^{d_\rho + 1} \cap W_0 \neq 0$.
Again using the $\rho$-equivariance of $\zeta$, we have that for all $\gamma \in \Gamma$, the set 
$\left\{ \zeta(x) \in \zeta(\partial_\infty \Gamma) : \zeta(x)^{d_\rho + 1} \cap \rho(\gamma) W_0 \neq 0 \right\}$
is of full measure which implies that their intersection is non-empty, and we can take $V = \zeta(x)^{d_\rho+1}$ for any $x \in \partial_\infty \Gamma$ such that $\zeta(x)$ lies in this intersection.

While the arguments until this point are valid, to conclude, the authors assert that this contradicts strong irreducibility by using \cite[Proposition 10.3]{labourie_anosov_2006}, which claims that if  the identity component $G_0$ of an algebraic group $G$ acts irreducibly on $\mathbb R^d$, then it can move any pair of subspaces of complementary dimensions into transverse position, i.e.\ for all $V \in \mathcal G_k$ and $W \in \mathcal G_{d-k}$, there exists $g \in G_0$ such that $gV \cap W = 0$.
However, this is false (see \cref{sec:adjoint_representations} and \cref{sec:so_p_p_action} for counterexamples) and consequently the claim that $\nu$ is irreducible need not be true.

Nevertheless, in \cref{sec:critical_exponent_calculation} we   calculate the dimension of the respective limit set and the critical exponent of the Falconer functional for the counterexamples of \cref{sec:adjoint_representations}, and observe that they do coincide. Hence they do not provide a counterexample to the main result found in \cite{pozzetti_anosov_2023}, so it is reasonable to expect that the main result of \cite{pozzetti_anosov_2023} is still valid (at least) under stronger assumptions.
In the following subsections we show that such assumptions can be Zariski-density in $\SL(d, \mathbb R)$ or $\SO(p,q)$ with $p \neq q$, or being $\mathrm{AdS}$-quasi-Fuchsian, which are all sufficient to ensure that the measure $\nu$ is irreducible and thus that the lower bound of \cref{thm:main} holds.

\subsection{\texorpdfstring{Proof of \cref{thm:main}}{Proof of Theorem \ref{thm:main}}}\label{sec:proof_of_theorem}
Here we finish the proof of \cref{thm:main} by considering distinct cases of representations which can move subspaces of complementary dimensions into transverse position.
An evident candidate is the class of representations with Zariski-dense image in $\SL(d, \mathbb R)$, considered in \cref{sec:density_in_sl}.
A more subtle case to consider in the same spirit is the case of representations with Zariski-dense image in $\SO(p,q)$, with $p \neq q$, which we show to be appropriate in \cref{sec:density_in_so}.
Finally, in \cref{sec:ads_quasi_fuchsian}, we use the classification of $\mathrm{AdS}$-quasi-Fuchsian representations in \cite{glorieux2018regularity} to recover the result for this specific case as well.

\subsubsection[\texorpdfstring{Case of Zariski density in $\SL(d,\mathbb R)$}{Case of Zariski density in SL(d,\mathbb R)}]{Case of Zariski density in $\SL(d,\mathbb R)$}\label{sec:density_in_sl}
Having outlined the proof of \cref{thm:main}, it is now rather quick to deduce that for a Zariski-dense representation $\rho: \Gamma \to \SL(d, \mathbb R)$, the \cref{eq:lower_bound} and hence the result of the theorem holds.
The missing ingredient is the following analogue of \cite[Lemma 6.8]{pozzetti_anosov_2023} for Zariski-dense representations:
\begin{lemma}[Zariski-dense representations in $\mathrm{SL}(d,\mathbb R)$ are $\mu$-irreducible]\label{lem:ZariskiDense}
    Let $\rho: \Gamma \to \SL(d,\mathbb R)$ be a linear representation and $\Theta \subseteq \Pi$.
    If $\rho(\Gamma)$ is Zariski-dense in $\SL (d, \mathbb R)$, then any $\rho(\Gamma)$-quasi-invariant measure $\mu$ on $\mathcal F_\Theta(\mathbb R^d)$ is irreducible. In particular any $(\rho(\Gamma), \phi)$-Patterson--Sullivan measure is irreducible.
\end{lemma}
\begin{proof}
    Since $\mu$ is quasi-invariant, for all $\rho(\gamma) \in \rho(\Gamma)$:
    \[
    \supp \mu = \supp \rho(\gamma)_* \mu = \rho(\gamma) \supp \mu.
    \]
    Thus the support of $\mu$ is $\rho(\Gamma)$-invariant: $\rho(\Gamma) \supp \mu = \supp \mu$.
    Taking the Zariski-closure of both sets, we see that $\mathrm{supp}\mu$ is Zariski dense because its Zariski-closure is $\SL(d, \mathbb R)$-invariant. In particular, $\mathrm{supp}(\mu)$ cannot be contained in $\mathrm{Ann}(y)$ for any $y \in \mathcal F_{\iota \Theta}(\mathbb R^d)$, because $\mathrm{Ann}(y)$ is a proper algebraic subvariety of $\mathcal F_\Theta(\mathbb R^d)$.
\end{proof}

\subsubsection[\texorpdfstring{Linear algebraic lemma in $\SO(p,q)$}{Linear algebraic lemma in SO(p,q)}]{Linear algebraic lemma in $\SO(p,q)$}\label{sec:linear_algebra}
In this subsection we prove \cref{lem:negative_determinant}, used in \cref{sec:density_in_so} to show that the Zariski-density in $\mathrm{SO}(p,q)$ is enough to give us the main result, when $p \neq q$.
More concretely, the lemma ensures that under minimal conditions, the stabilizer of a subspace in $\OO(p,q)$ contains an element of negative determinant.
With this, we obtain in \cref{lem:transverse_via_O} that given any two subspaces of complementary dimensions, one can move to a subspace  transverse to the other using a transformation in $\SO(p,q)$.

\begin{restatable}[Stabilizer with negative determinant]{lemma}{negativeDeterminant}\label{lem:negative_determinant}
    Let $V$ be a $k$-dimensional subspace of $\mathbb R^{p,q}$.
    If any of the following three conditions is satisfied, 
    \begin{enumerate}
        \item $V$ is non-degenerate,
        \item $V$ is degenerate but not maximal isotropic,
        \item $p \neq q$,
    \end{enumerate}
    then there exist transformations that stabilize $V$ and have determinants $\pm 1$, i.e.
    \[
        \mathrm{St}_{O(p,q)}(V) \cap \mathrm{SO}(p,q), \mathrm{St}_{O(p,q)}(V) \cap \OO^-(p,q) \neq \emptyset.
    \]
\end{restatable}
\begin{proof}
    \begin{enumerate}
        \item If $V$ is non-degenerate, in an appropriate basis the form is given by the matrix $I_{p,q}$, where $(p,q)$ is its signature.
        Then, any transformation with diagonal entries $\pm 1$ belongs to $\OO(V)$, and clearly can be chosen to have either positive or negative determinant, which can be preserved after extending to the whole space $\mathbb R^{p,q}$.
        \item If $V$ is degenerate but not maximal isotropic, we distinguish two subscases, whether $V$ is totally isotropic, or not.
        \begin{enumerate}[label=(\roman*)]
            \item If $V$ is totally isotropic, there exists a totally isotropic subspace $W$ transverse to $V$, with the same dimension as $V$, and such that the form is non-degenerate when restricted to $V \oplus W$.
        To see this, we can proceed by induction on $k = \dim V$.
        If $k = 1$ and $V= \mathbb R e_1$, non-degeneracy of the form implies that there exists some $e_1' \in V^c$ such that $\langle e_1', e_1 \rangle = 1$.
        Then $f_1 = e_1' - (\langle e_1', e_1' \rangle/2) e_1$ is isotropic, does not belong to $V$ and satisfies $\langle f_1, e_1 \rangle = 1$.
        If $k > 1$, assume that $e_1, \cdots, e_k$ is a basis of $V \cap V^\perp$, and $f_1, \cdots, f_{k-1}$ span an isotropic subspace $W_{k-1}$ such that the form has matrix
        \[
        \begin{pmatrix}
        0 & I_{k-1} \\
        I_{k-1} & 0
        \end{pmatrix}
        \]
        in the basis $e_1, \cdots, e_{k-1}, f_1, \cdots, f_{k-1}$.
        Then there exists some $e_k' \in V^c$ such that $\langle e_k', e_j \rangle = \langle e_k', f_j \rangle = 0$ for $j = 1, \cdots, k-1$ and $\langle e_k', e_k \rangle = 1$.
        Letting $f_k = e_k' - (\langle e_k', e_k' \rangle/2) e_k$, we have that $f_k$ is isotropic and satisfies $\langle f_k, e_j \rangle = \delta_{kj}$ for $j = 1, \cdots, k$ and $\langle f_k, f_j \rangle = 0$ for $j = 1, \cdots, k-1$.

        Since we have assumed that $V$ is not maximal isotropic, we have $\dim(V) = \dim(W) < \min\{p,q\}$, so that $\dim(V \oplus W) < p + q$.
        In particular, the subspace $(V \oplus W)^\perp$ is non-trivial, and non-degenerate.
        In a basis adapted to the splitting $\mathbb R^{p,q} = V \oplus W \oplus (V \oplus W)^\perp$, we may consider any transformation $g = I_{\dim(V)} \oplus I_{\dim(W)} \oplus B$, with $B \in \OO((V \oplus W)^\perp)$.
        By the remark at the beginning of the proof, we can choose $B$ to have either positive or negative determinant, and $g$ stabilizes $V$ by construction.
        \item If $V$ is not totally isotropic, we let $V_\pm$ be two maximal positive and negative definite subspaces of $V$ respectively, and $V_0$ be a maximal isotropic subspace of $V$ such that $V = V_- \oplus V_0 \oplus V_+$.
        Without loss of generality, we can assume that $\dim V_- > 0$, treating the case $\dim V_+ > 0$ in the same way.
        Any transformation of the form $g = A \oplus I_{\dim V_0} \oplus I_{\dim V_+}$ with $A \in \OO(V_-)$ lies in $\mathrm{St}_{\OO(p,q)}(V)$.
        Again by the remark at the beginning of the proof, we can choose $A$ to have either positive or negative determinant.
        \end{enumerate}
        
        \item If $p \neq q$, we only need to consider the case where $V$ is a maximal isotropic subspace, because the other cases are covered by the previous two items.
        In this case, we proceed as in the second one, and find a transverse isotropic subspace $W$ of the same dimension as $V$ such that the form is non-degenerate when restricted to $V \oplus W$.
        Because $p \neq q$, we have again that $\dim(V \oplus W) = 2 \min\{p,q\} < p + q$, so that $(V \oplus W)^\perp$ is non-trivial and non-degenerate.
        Then we conclude as in the first item, by considering transformations of the form $g = I_{\dim(V)} \oplus I_{\dim(W)} \oplus B$, with $B \in \OO((V \oplus W)^\perp)$, and choosing $B$ to have either positive or negative determinant.
    \end{enumerate}

\end{proof}

\subsubsection[Case of Zariski density in SO(p,q)]{Case of Zariski density in $\SO(p,q)$}\label{sec:density_in_so}
The goal of this section is to show that if we assume (\ref{item:main_desnse_so}) that the representation is Zariski-dense in $\SO(p,q)$ with $p \neq q$, then the proof of the lower bound for the dimension of the limit set in \cite{pozzetti_anosov_2023} is still valid.
In light of \cref{sec:proof_strategy}, this will follow from an analogue of \cite[Proposition 10.3]{labourie_anosov_2006} for Zariski-dense representations in $\SO(p,q)$.
To obtain this, we will use \cref{lem:negative_determinant}.

The only ingredient missing to obtain the result in the case of a Zariski-dense representation in $\SO(p,q)$ with $p \neq q$ is the following lemma.
\begin{lemma}\label{lem:transverse_via_O}
    Let $k+l \leq d, p+q = d, p\neq q, C \in \mathcal G_k(\mathbb R^{p,q}), B \in \mathcal G_l(\mathbb R^{p,q})$.
    Then there exists some $g \in \SO(p,q)$ such that $gC \cap B = \{0\}$.
\end{lemma}
\begin{proof}
    By \cref{lem:negative_determinant}, there exists some transformation $h \in \OO(p,q)$ with determinant $\det(h) = -1$ and which stabilizes $C$.
    It suffices now to show that we may also find some $g \in \OO(p,q)$ such that $g C \cap B = \{0\}$.
    Indeed, then both $g$ and $gh$ are in $\OO(p,q)$ and move $C$ to a subspace transverse to $B$.
    On the other hand they have opposite determinants. so one of them must be in $\SO(p,q)$.

    For a subspace $C' \in \mathcal G_k(\mathbb R^{p,q})$, we let $a, c$ be the dimensions of any of its maximal positive and negative subspaces of $C$ respectively, and $b = \mathrm{dim}(C' \cap (C')^\perp)$.
    By Witt's theorem, it suffices to show the existence of some subspace $C' \in \mathcal G_k(\mathbb R^{p,q})$ that is transverse to $B$ and for which $(a_{C'}, b_{C'}, c_{C'}) = (a_C, b_C, c_C)$.
    We will proceed by induction on $b_C$.

    If $b_C = 0$, then $C$ is non-degenerate and we can perturb it into a subspace transverse to $B$ without changing the signature.
    Clearly the same argument treats the case where $B$ is non-degenerate.
    
    Assuming now neither $B$ nor $C$ is non-degenerate, we claim that there exists some $v \in B^c$ that is isotropic and not contained in $B^\perp$.
    Indeed, let $e_1 \in B \cap B^\perp$ be a nonzero totally isotropic vector.
    Because $\mathbb R^{p,q}$ is non-degenerate, there exists some $e_1' \in \mathbb R^{p,q}$ for which $\langle e_1, e_1' \rangle = 1$.
    Then the vector $v = e_1' - \lambda e_1$ is isotropic for $\lambda = \langle e_1', e_1' \rangle/2$, and is not contained in $B^\perp$ since $\langle e_1, v \rangle = 1$.

    Fixing such an isotropic vector $v \in B^c \cap (B^\perp)^c$, we consider the projection $p: v^\perp \to v^\perp/\mathbb R v$, which is form-preserving.
    Here $v^\perp$ stands for the subspace of vectors orthogonal to $\mathbb Rv$, and $v^\perp/ \mathbb Rv$ is equipped with the induced non-degenerate form of signature $(p-1, q-1)$.
    Then $\tilde B = p(B \cap v^\perp) \in G_{l-1}(v^\perp/ \mathbb Rv)$ because:
    \[
    \dim(\tilde B) = \dim(B \cap v^\perp) - \dim(B \cap \mathbb Rv) = \dim(B \cap v^\perp) = l-1.
    \]
    Since $(k-1) + (l-1) \leq d-2$, the induction hypothesis implies that there exists some $\tilde C \in \mathcal G_{k-1}(v^\perp/\mathbb Rv)$ that is transverse to $\tilde B$ and has $(a_{\tilde C}, b_{\tilde C}, c_{\tilde C}) = (a_{C}, b_{C}-1, c_{C})$.
    We let $C' = p^{-1}(\tilde C)$, and we claim that it is transverse to $B$ and satisfies $(a_{C'}, b_{C'}, c_{C'}) = (a_C, b_C, c_C)$.
    
    To see that $C'$ is transverse to $B$, we note that $C' \cap B = C' \cap B \cap v^\perp \subseteq p^{-1}(\tilde C \cap \tilde B) = \mathbb Rv$, but no non-zero element of $\mathbb Rv$ is contained in $B$ by construction.
    On the other hand, let $\tilde u_1, \cdots u_{a_C}, \tilde v_1, \cdots, \tilde v_{b_C-1}, \tilde w_1, \cdots, \tilde w_{c_C}$ be a basis of $\tilde C$ such that the form is positive definite on the span of the $\tilde u_i$'s, negative definite on the span of the $\tilde w_i$'s, and zero on the span of the $\tilde v_i$'s.
    Letting $u_i, v_i, w_i$ be any nonzero lifts of $\tilde u_i, \tilde v_i, \tilde w_i$ to $C'$, we can build a basis 
    \[
    u_1, \cdots u_{a_C}, v, v_1, \cdots, v_{b_C-1}, w_1, \cdots, w_{c_C} \text{ of } C',
    \]
    for which the form is positive definite on the span of the $u_i$'s, negative definite on the span of the $w_i$'s, and zero on the span of the $v_i$'s and $v$.
\end{proof}
\subsubsection{Case of AdS-quasi-Fuchsian representations}\label{sec:ads_quasi_fuchsian}
We will now see how the final assumption of \cref{thm:main} is sufficient to avoid the use of \cite[Proposition 10.3]{labourie_anosov_2006} in the proof of the bound \ref{eq:lower_bound}.
For this, we only need to use the following result of \cite{glorieux2018regularity} which shows that strongly irreducible AdS-quasi-Fuchsian representations are Zariski-dense in $\SO(p,2)$.
\begin{proposition}[Proposition 1.4 in \cite{glorieux2018regularity}]
    Let $\rho: \Gamma \to \SO(p,2)$ be an AdS-quasi-Fuchsian representation that does not preserve any totally geodesic copy of $\mathbb H^p$ in $\mathrm{AdS}^{p}$.
    Then $\rho(\Gamma)$ is Zariski-dense in $\SO(p,2)$.
\end{proposition}
Hence, the result of \cref{sec:density_in_so} may be used to show that the lower bound \cref{eq:lower_bound} holds, when $p \neq 2$.
However, the theorem has been proven for $p=2$ in \cite[Theorem 1.6 and paragraphs following Question 1.4]{glorieux_hausdorff_2023}. 
There, they argue that for $\mathbf H^{p,q+1}$-convex-cocompact representations, the Hausdorff dimension of $\xi^1(\partial_\infty \Gamma)$ is the same as the Hausdorff dimension of $(\xi^1, \xi^{d-1})(\partial_\infty \Gamma)$, so their Theorem 1.1 implies that the critical exponent of the functional $\alpha_{1,2} = J^u_1$ is equal to the Hausdorff dimension $\dim_\mathcal H(\xi^1(\partial_\infty \Gamma))$ of the limit set.
The latter is contained in the boundary of $\mathrm{AdS}^2$, so it is at most 1.

\section{Equivariant spacelike embeddings}\label{sec:equivariant_spacelike_embeddings}
In this section, we give a counterexample to \cite[Lemma 6.8]{pozzetti_anosov_2023} and \cite[Proposition 10.3]{labourie_anosov_2006} that appear in the proof of \cite[Theorem A]{pozzetti_anosov_2023}.
Using the theory of Hermitian forms recalled in \cref{sec:hermitian_forms}, we prove in \cref{sec:adjoint_representations} that the adjoint representation of the automorphism group of a non-degenerate Hermitian form over any of these fields gives rise to a spacelike embedding of the corresponding symmetric space into $\mathbf H^{p',q'-1}$ for some $p',q'$.
In the case of forms with signature $(n,1)$, the resulting representation is $\mathbf H^{p',q'-1}$-convex-cocompact of maximal virtual cohomological dimension (in particular it is $P_1$-Anosov).
Thus its limit set is a Lipschitz (even smooth for our examples) sphere of dimension $p'-1$, which places us under the setting of \cite{pozzetti_anosov_2023}.
We also give a proof that the tangent spaces of the limit set intersect non-trivially, giving us a counterexample to \cite[Lemma 6.8]{pozzetti_anosov_2023} and \cite[Proposition 10.3]{labourie_anosov_2006}.
However, in \cref{sec:critical_exponent}, we show that the equality of the critical exponent of the Falconer functional and the Hausdorff dimension still holds, meaning that this family of examples does not contradict the main result of \cite{pozzetti_anosov_2023}.

\subsection{Pseudo-Riemannian hyperbolic spaces}\label{sec:pseudo_riemannian_hyperbolic_spaces}
A vast class of examples where our results apply and counterexamples arise is that of representations into the isometry group of a pseudo-Riemannian hyperbolic space $\mathbf H^{p,q}$.
We begin by recalling their definition and a useful coordinate system for its two-fold cover $\hat{\mathbf H}^{p,q}$.
Using these coordinates, we provide a criterion to ensure that a spacelike immersion is actually an embedding. This will be used in \cref{prop:adjoint_representation_convex_cocompact} to prove that for a uniform lattice in the automorphism group of a non-degenerate Hermitian form over the space of traceless Hermitian forms, its action admits an equivariant spacelike embedding into a pseudo-Riemannian hyperbolic space.

\subsubsection{Definitions and coordinates}\label{sec:pseudo_riemannian_hyperbolic_spaces_definitions}
\begin{definition}[Pseudo-Riemannian hyperbolic spaces]
    Let $p,q \in \mathbb N$ with at least one of them non-zero.
    \begin{enumerate}[label=(\roman*)]
        \item The \emph{Minkowski space} $\mathbb R^{p,q+1}$ of signature $(p,q+1)$ is the vector space $\mathbb R^{p+q+1}$ equipped with a non-degenerate symmetric bilinear form $\langle \cdot, \cdot \rangle_{p,q+1}$ of signature $(p,q+1)$.
        \item The \emph{pseudo-Riemannian hyperbolic space} $\mathbf H^{p,q}$ is the set of negative lines in the projectivisation $\mathbb P(\mathbb R^{p,q+1})$ of the Minkowski space, i.e.\ 
        \[
        \mathbf H^{p,q} = \{ \mathbb R x \in \mathbb P(\mathbb R^{p,q+1}) : \langle x,x \rangle_{p,q+1} < 0 \}.
        \]
        \item The \emph{two-fold cover} $\hat{\mathbf H}^{p,q}$ of $\mathbf H^{p,q}$ is the set of unit negative vectors in $\mathbb R^{p,q+1}$, i.e.\
        \[
        \hat{\mathbf H}^{p,q} = \{ x \in \mathbb R^{p,q+1} : \langle x,x \rangle_{p,q+1} = -1 \}.
        \]
    \end{enumerate}
    We shall denote with 
    \[
    \pi: \hat{\mathbf H}^{p,q} \to \mathbf H^{p,q}
    \]
    the 2-to-1 covering map that sends a vector to the line it spans.
\end{definition}

To lay the stage for the next proposition, let $E$ be a spacelike $p$-dimensional subspace of $\mathbb R^{p,q+1}$, $F = E^\perp$ its orthogonal complement.
We fix a basis of $\mathbb R^{p,q+1} = E \oplus F$, for which the bilinear form $\langle \cdot, \cdot \rangle_{p,q+1}$ can be expressed as
\[
\langle u_E + u_F, v_E + v_F \rangle_{p,q+1} = u_E^1 v_E^1 + \cdots + u_E^p v_E^p - u_F^1 v_F^1 - \cdots - u_F^{q+1} v_F^{q+1},
\]
for $u_E, v_E \in E$ and $u_F, v_F \in F$.
In this way, restricting the ambient form $\langle \cdot, \cdot \rangle_{p,q+1}$ to $E$ and $F$, we identify $E \simeq \mathbb R^p$ with the Euclidean space and $F \simeq \mathbb R^{q+1}$ with the Minkowski space of signature $(0,q+1)$.
We will denote with $\mathbb B^p$ the open unit ball in $E$ with respect to the induced Euclidean metric, and with $\mathbb S^q$ the pseudo-sphere in $F$ defined as
    \begin{align*}
        \mathbb B^p &= \{ u \in E : \langle u,u \rangle_{p,q+1} < 1 \},\\
        \mathbb S^q &= \{ v \in F : \langle v,v \rangle_{p,q+1} = -1 \}.
    \end{align*}
\begin{lemma}[Warped product structure for $\hat{\mathbf H}^{p,q}$]\label{lem:warped_product_structure}
    The map
    \begin{align*}
        \Psi: \mathbb B^p \times \mathbb S^q &\to \hat{\mathbf H}^{p,q}\\
        (u,v) &\mapsto \frac{2}{1 - \|u\|^2} u + \frac{1 + \|u\|^2}{1 - \|u\|^2} v,
    \end{align*}
    is a diffeomorphism onto $\hat{\mathbf H}^{p,q}$, where we denote $\|u\|^2 = \langle u,u \rangle_{p,q+1}$.
    Moreover, the pull-back metric $\Psi^* g_{\hat{\mathbf H}^{p,q}}$ is given by
    \[
    \Psi^* g_{\hat{\mathbf H}^{p,q}} = \frac{4}{(1 - \|u\|^2)^2} g_{\mathbb B^p} \oplus - \left(\frac{1 + \|u\|^2}{1 - \|u\|^2}\right)^2 g_{\mathbb S^q},
    \]
    where $g_{\mathbb B^p}$ is the flat metric on $\mathbb B^p$ and $g_{\mathbb S^q}$ is the standard Riemannian spherical metric on $\mathbb S^q$.
    In particular, $g_{\mathbb B^p}$ is the metric induced by the restriction of $\langle \cdot, \cdot \rangle_{p,q+1}$ to $\mathbb B^p$, and $g_{\mathbb S^q}$ is the metric induced by the restriction of $-\langle \cdot, \cdot \rangle_{p,q+1}$ to $\mathbb S^q$.
\end{lemma}
The proof of the above lemma is the same as the one of \cite[Proposition 3.5]{collier2019geometry}, so we omit it here.

\subsubsection{Immersions that are actually embeddings}
Recall that an immersion $f: M \to N$ between a manifold $M$ and a pseudo-Riemannian manifold $(N,g_N)$ is spacelike if the pull-back metric $f^* g_N$ is Riemannian.
In the construction of our counterexamples (more specifically in \cref{prop:adjoint_representation_convex_cocompact}), we will need a criterion to ensure that a spacelike immersion is actually an embedding.
This is the goal of the following lemma, whose proof is a straightforward generalization of \cite[Lemma 3.7]{collier2019geometry} from the setting of surfaces to higher dimensional manifolds.
Its proof can be found in \cite[Lemma 3.11]{seppi2023complete}.
\begin{lemma}[Embedding criterion for immersions]\label{lem:spacelike_embedding}
    Let $M$ be a connected manifold of dimension $p$ and $f: M \to \mathbf H^{p,q}$ be an injective spacelike immersion.
    If the pull-back metric $f^* g_{\mathbf H^{p,q}}$ is complete, then $f$ is an embedding and its image $f(M)$ is an embedded spacelike $p$-dimensional submanifold of $\mathbf H^{p,q}$.
\end{lemma}

Specializing to the case of immersions of symmetric spaces, we obtain the following criterion for embeddedness that is more tailored to the setting of our counterexamples.
In fact, the following proof would still be valid for any quotient of $G$ by a closed subgroup that admits cocompact quotients.
\begin{lemma}[Cocompact embedding of symmetric spaces criterion]\label{lem:criterion_spacelike_cocompact}
        Let $G$ be a semisimple Lie group, $K \leq G$ be a compact subgroup and $f: G/K \to \mathbf H^{p,q}$ be a spacelike immersion. Denote with $\tilde M = f(G/K)$ the image of $f$.
        \begin{enumerate}
            \item If $\rho: G \to \mathrm{SO}(p,q+1)$ is an injective representation such that $f$ is $\rho$-equivariant, then $f$ is a proper embedding and $\tilde M$ is a complete connected spacelike submanifold of $\mathbf H^{p,q}$.
            \item If $\Gamma \leq G$ is a uniform lattice of $G$, and $\rho: \Gamma \to \mathrm{SO}(p,q+1)$ is an injective representation such that $f$ is $\rho$-equivariant, then $f$ is a proper embedding and $\tilde M$ is a complete connected spacelike submanifold of $\mathbf H^{p,q}$ over which $\rho(\Gamma)$ acts properly discontinuously and cocompactly.
        \end{enumerate}
\end{lemma}
\begin{proof}
    Since $G$ is a semisimple Lie group, a uniform lattice $\Gamma$ exists.
    Therefore the first point of the lemma is a special case of the second.
    To prove the latter, we will show that the pull-back metric $f^* g_{p,q}$ is complete, which will imply by the preceding lemma that $f$ is a proper embedding and that $\tilde M$ is a spacelike (embedded) submanifold.

    First note that since $f$ is a spacelike immersion, $f^* g_{p,q}$ is a Riemannian metric on $G/K$.
    A priori, $\Gamma$ may not act by isometries with respect to this metric, but $\rho$-equivariance of $f$ implies that it does.
    Indeed, writing the equivariance condition as $f \circ \gamma = \rho(\gamma) \circ f$ for all $\gamma \in \Gamma$, we have that $\gamma^* (f^* g_{p,q}) = f^* (\rho(\gamma)^* g_{p,q}) = f^* g_{p,q}$, since $\rho(\gamma) \in \SO(p,q+1)$ is an isometry of $\mathbf H^{p,q}$.
    
    Since $\Gamma$ is a discrete subgroup of isometries of a Riemannian metric on $G/K$, which is a complete Riemannian manifold, we know that $\Gamma$ acts properly discontinuously.
    Moreover, since $\Gamma$ is a uniform lattice in $G$, the quotient $\Gamma \backslash G/K$ is compact.
    
    Knowing that $\Gamma$ acts smoothly, without torsion,
    properly discontinuously and cocompactly on the Riemannian manifold $(G/K, f^* g_{p,q})$, we may conclude using the theorem of Hopf--Rinow that $f^* g_{p,q}$ is a complete Riemannian metric on $G/K$.
    Indeed, $(\Gamma \backslash G/K, f^* g_{p,q})$ inherits a Riemannian metric that makes the quotient map a local isometry.
    Since the quotient is compact, the metric on it is complete, implying that $f^* g_{p,q}$ is complete as well.

    Because $f: G/K \to f(M) \subseteq \mathbf H^{p,q}$ is a $\rho$-equivariant diffeomorphism, $\rho(\Gamma)$ is a discrete torsion-free subgroup of isometries of $f(M)$, so its action is properly discontinuous.
    Equivariance also implies that it descends to a diffeomorphism $\Gamma \backslash G/K \simeq f(M)/\rho(\Gamma)$ that makes the following diagram commute: 
    \[
        \begin{tikzcd}
            G/K \arrow[r, "f"] \arrow[d] &f(M) \subseteq \mathbf H^{p,q} \arrow[d] \\
            \Gamma \backslash G/K \arrow[r, dashed]  & \rho(\Gamma) \backslash f(M)
        \end{tikzcd}.
    \]
    Hence $\rho(\Gamma) \backslash f(M)$ is compact, being the continuous image of the compact $\Gamma \backslash G/K$.
\end{proof}

\subsection{Hermitian forms}\label{sec:hermitian_forms}
Here we recall basic definitions and properties of Hermitian forms over the real numbers $\mathbb R$, the complex numbers $\mathbb C$ and the quaternions $\mathbb H$.
We then describe how fixing a basis yields an identification between the automorphism group of a Hermitian form, the space of transformations whose adjoint is equal to their inverse and a subspace of matrices over $\mathbb F$.
Similarly, fixing a basis and a "reference" Hermitian form, we can establish identifications between the space of Hermitian forms, the space of self-adjoint transformations and a subspace of matrices over $\mathbb F$.
These identifications will be useful in computations when handling representations (as in \cref{sec:adjoint_representations}).
\subsubsection{Hermitian forms and their automorphism groups}\label{sec:hermitian_forms_definitions}
We recall that a Hermitian form over a complex vector space is a bilinear form that is linear in the second argument and conjugate linear in the first, and that is conjugate symmetric.
This definition has a natural generalization to the case of quaternionic vector spaces, with the only difference arising from the non-commutativity of the quaternions.
In the case of $\mathbb F = \mathbb R$ of the following definition, conjugation is the identity.
\begin{definition}[Hermitian form]\label{def:hermitian_form}
    Let $V$ be a finite dimensional vector space over a field $\mathbb F$ that is either $\mathbb R$, $\mathbb C$ or $\mathbb H$.
    A \emph{Hermitian form} on $V$ is an $\mathbb R$-bilinear form $h: V \times V \to \mathbb F$ such that
    \begin{enumerate}[label=(\roman*)]
        \item It is linear in the second argument and conjugate linear in the first: $h(x \lambda, y \mu) = \bar \lambda h(x,y) \mu$ for all $x,y \in V$ and $\lambda, \mu \in \mathbb F$.
        \item It is conjugate symmetric: $h(x,y) = \overline{h(y,x)}$ for all $x,y \in V$.
    \end{enumerate}
    The \emph{space of Hermitian forms} on $V$ is denoted with $\mathrm{Herm}_\mathbb F(V)$, on which $\GL(V)$ acts by 
    \[
    (g \cdot h) (u, v) = h(g^{-1}u, g^{-1}v), \text{ for } g \in \GL(V),\ h \in \mathrm{Herm}_\mathbb F(V),\ u,v \in V.
    \]
    We shall call a form $h \in \mathrm{Herm}_\mathbb F(V)$ \emph{non-degenerate} if the induced map $V \to V^*$ given by $v \mapsto h(v, \cdot)$ is an isomorphism.
\end{definition}

\begin{definition}[Adjoint transformation]
    Let $V$ be a finite dimensional vector space over $\mathbb F \in \{ \mathbb R, \mathbb C, \mathbb H \}$, and $h \in \mathrm{Herm}_\mathbb F(V)$ be a non-degenerate Hermitian form on $V$.
    For every linear transformation $A \in \mathfrak{gl}(V)$, we denote with $A^{*_h} \in \mathfrak{gl}(V)$ the \emph{adjoint transformation} of $A$ with respect to $h$, defined by the relation
    \[
    h(Au, v) = h(u, A^{*_h}v), \text{ for all } u,v \in V.
    \]
    The induced mapping $\mathfrak{gl}(V) \to \mathfrak{gl}(V)$ is an $\mathbb R$-linear involution.
    If we fix a basis $e_1, \cdots, e_n$ of $V$, in the usual identification of $\gl(V)$ with $\mathbb F^{n \times n}$ the involution is given by $A \mapsto H^{-1} A^* H$, where $H$ is the matrix of the Hermitian form $h$ with respect to the basis $e_1, \cdots, e_n$, i.e.\ 
    \[
    h_0 \left( \sum_i x_i e_i, \sum_j y_j e_j \right) = \sum_{i=1}^n \bar x_i H_{ij} y_j,
    \]
    and $A^* = \overline{A^t}$ is the adjoint of $A$ with respect to the standard Hermitian form on $\mathbb F^{n \times n}$ (whose corresponding matrix is the identity).
\end{definition}
\begin{definition}[Automorphism group of a Hermitian form]
    Let $V$ be a finite dimensional vector space over $\mathbb F \in \{ \mathbb R, \mathbb C, \mathbb H \}$, and $h \in \mathrm{Herm}_\mathbb F(V)$ be a non-degenerate Hermitian form on $V$.
    The \emph{automorphism group} of $h$ is defined in the three equivalent ways:
    \begin{align*}
        \mathrm{Aut}(h) &= \mathrm{Stab}_{\GL(V)}(h)\\
        &= \left\{ g \in \GL(V) : h(g u, g v) = h(u,v) \text{ for all } u,v \in V \right\}\\
        &= \left\{ g \in \GL(V) : g^{*_h} = g^{-1} \right\}.
    \end{align*}
\end{definition}
If we fix a basis $e_1, \cdots, e_n$ of $V$, then the usual identification of linear transformations with matrices restricts to an identification of the automorphism group of $\mathrm{Aut}(h)$ with a subspace of $\mathbb F^{n\times n}$:
\[
\mathrm{Aut}(h) \simeq \left\{ A \in \mathbb F^{n \times n}: H^{-1} A^* H = A^{-1}
    \right\}.
\]
It will be useful for us to fix a non-degenerate form $h_0 \in \mathrm{Herm}_\mathbb F(V)$, in order to identify the space of Hermitian forms $\mathrm{Herm}_\mathbb F(V)$ with the space of self-adjoint linear transformations:
\begin{lemma}[Hermitian forms as self-adjoint transformations]\label{lem:hermitian_forms_self_adjoint_representations}\label{lem:identifications}
    Let $V$ be a finite dimensional vector space over $\mathbb F \in \{ \mathbb R, \mathbb C, \mathbb H \}$, and fix a non-degenerate Hermitian form $h_0 \in \mathrm{Herm}_\mathbb F(V)$ on $V$.
    Then we can identify (through a conjugate-linear isomorphism) the space of Hermitian forms $\mathrm{Herm}_\mathbb F(V)$ with the space of self-adjoint linear transformations, as representations of the automorphism group $\mathrm{Aut}(h_0)$:
    \begin{align*}
        \left\{ T \in \mathfrak{gl}(V) : T^{*_{h_0}} = T \right\} &\cong \mathrm{Herm}_\mathbb F(V)\\
        T &\mapsto h_0(T \cdot, \cdot)
    \end{align*}
    Given a basis $e_1, \cdots, e_n$ of $V$, the usual identification of transformations with matrices gives us an isomorphism of representations:
    \begin{align*}
        \left\{ T \in \gl(V) : T^{*_{h_0}} = T \right\} &\cong \left\{ A \in \mathbb{F}^{n \times n} : H_0^{-1} A^* H_0 = A\right\}
    \end{align*}
    Both of these identifications are isomorphisms of representations of $\mathrm{Aut}(h_0)$, which acts by 
    \begin{align*}
        (g \cdot h)(u, v) &= h(g^{-1} u, g^{-1} v), \text{ on } \mathrm{Herm}_\mathbb F(V)\\
        (g \cdot T) &= g T g^{-1}, \text{ on } \left\{T \in \gl(V): T^{*_{h_0}} = T\right\}\\
        (g \cdot A) &= g A g^{-1}, \text{ on } \left\{ A \in \mathbb F^{n\times n} : H_0^{-1} A^* H_0 = A \right\},
    \end{align*}
    where in the right hand side of the last equality, we denote with the same symbol $g$ the automorphism itself (as an element $\Aut(h_0) \subseteq \gl(V)$), and its matrix representation with respect to the basis $e_1, \cdots, e_n$.
\end{lemma}
\begin{proof}
    If $T \in \gl(V)$, then the form $h_0(T \cdot, \cdot)$ is a Hermitian form if and only if $T^{*_{h_0}} = T$.
    Hence the mapping $T \mapsto h_0(T \cdot, \cdot)$ is well-defined, and by non-degeneracy, it is injective.

    For surjectivity, we first remark that non-degeneracy of $h_0$ is equivalent to having the mapping $\Phi_{h_0}: V \to V^*, v \mapsto h_0(\cdot, v)$ be an isomorphism.
    Thus, $T \defeq \Phi_{h_0}^{-1} \circ \Phi_h :V \to V$ is a well-defined linear transformation for any Hermitian form $h$, and by construction we have that $h(x, \cdot) = h_0(T x, \cdot)$ for all $x \in V$.
\end{proof}
While the above identification guarantees the existence of some $T \in \mathfrak{gl}(V)$ such that $h(x,y) = h_0(T x, y)$ for all $x,y \in V$, its proof does not give us a way to explicitly construct it.
However, this is possible by fixing a basis of $V$ and using matrix representations, thus providing an alternative proof of surjectivity.
\begin{remark}[Construction of $T$]
    Let $h \in \mathrm{Herm}_\mathbb F(V)$ be a Hermitian form, $e_1, \cdots, e_n$ be a basis of $V$ and $H, H_0$ be the matrices of $h$ and $h_0$ with respect to this basis.
    Then, $h = h_0(T \cdot , \cdot)$ is equivalent to $T = H_0^{-1} H$, where here we use the same notation $T$ to denote the transformation appearing in \cref{lem:hermitian_forms_self_adjoint_representations} and its corresponding matrix with respect to the fixed basis.
\end{remark}

\begin{remark}[Choice of reference form and basis]\label{rem:reference_form_and_basis}
    We would like to stress that while the matrix representation of the automorphism group $\mathrm{Aut}(h_0)$ depends only on the choice of basis of $V$ (in contrast to $\mathrm{Herm}_\mathbb F(V)$ that depends on the choice of a reference form as well), the action of $\mathrm{Aut}(h_0)$ on the subspace of $\mathbb F^{n\times n}$ associated to $\mathrm{Herm}_\mathbb F(V)$ is well-defined only when the reference form is $h_0$.
    Similarly, for the definition in \cref{sec:hyperbolic_spaces} of the hyperbolic space $\mathbf H^n_\mathbb F$, the same form must be used.
\end{remark}

Using the identifications of \cref{lem:identifications}, we can define traceless Hermitian forms, which will be useful in the construction of our counterexamples.
For this, we introduce the following notion of trace for real, complex and quaternionic matrices:
\begin{definition}[Reduced trace of quaternionic matrix]
    Let $V$ be a vector space over $\mathbb F \in \{ \mathbb R, \mathbb C, \mathbb H \}$.
    \begin{enumerate}[label=(\roman*)]
        \item We define the \emph{(reduced) trace of a matrix} $A \in \mathbb F^{n \times n}$ as the real part of the sum of its diagonal entries, i.e.\
        \[
        \tr_\mathbb R(A) = \mathrm{Re}\left( \sum_i A_{ii} \right).
        \]
        \item We define the \emph{(reduced) trace of a linear transformation} $T \in \gl(V)$ as the reduced trace of its matrix representation with respect to any basis of $V$.
        \item The (reduced) trace of a Hermitian form $h \in \mathrm{Herm}_\mathbb F(V)$ is defined as the reduced trace of the corresponding self-adjoint transformation with respect to any basis of $V$ 
        \item The Lie algebra consisting of matrices with zero reduced trace is denoted with $\mathfrak{sl}(n, \mathbb H)$.
    \end{enumerate}
\end{definition}
The reason that we use the reduced trace and not the sum of the diagonal entries of the matrix representation of $T$ is that the non-commutativity of the quaternions renders it dependent on the choice of basis of $V$.
Using the above, we can define traceless Hermitian forms as follows:
\begin{definition}[Traceless Hermitian forms]
    Let $V$ be a finite dimensional vector space over $\mathbb F \in \{ \mathbb R, \mathbb C, \mathbb H \}$ and $h_0 \in \mathrm{Herm}_\mathbb F(V)$ be a non-degenerate Hermitian form on $V$.
    We define the subspace of \emph{traceless Hermitian forms} as 
    \[
    \mathrm{Herm}^0_\mathbb F(V) \defeq \begin{cases}
    \left\{ h \in \mathrm{Herm}_\mathbb F(V) : \tr(h) = 0 \right\}, & \text{if } \mathbb F \in \{\mathbb R, \mathbb C\}\\
    \left\{ h \in \mathrm{Herm}_\mathbb H(V) : \tr_\mathbb R(h) = 0 \right\}, & \text{if } \mathbb F = \mathbb H
    \end{cases}
    \]
    where the trace of a Hermitian form $h$ is defined as the trace of the corresponding self-adjoint transformation with respect to the reference form $h_0$.
\end{definition}

The following representation will be the object of study for the rest of this section.
For the right values of $p$ and $q$, in \cref{sec:adjoint_representations}  it will yield the examples announced in the introduction that admit equivariant spacelike embeddings in $\mathbb H^{p,q}$ and disprove the lemmas in \cite{pozzetti_anosov_2023}.
\begin{definition}[Representation $\tau$]\label{def:representation_tau}
Let $V$ be a finite dimensional vector space over $\mathbb F \in \{ \mathbb R, \mathbb C, \mathbb H \}$, and $h_{p,q}$ be a non-degenerate Hermitian form of signature $(p,q)$ on $V$.
denoting by $G = \mathrm{Aut}(h_{p,q}) \cap \mathrm{SL}(V, \mathbb F)$, we define the \emph{representation $\tau$}
\[
\tau: G \to \GL\left(\mathrm{Herm}^0_\mathbb F(V)\right)
\] 
to be the one corresponding to the action of $G$ on $\mathrm{Herm}^0_\mathbb F(V)$.
\end{definition}
Since the reduced trace of a Hermitian form is preserved by conjugation, the representation $\tau$ is well-defined, in the sense that it preserves the subspace of traceless Hermitian forms $\mathrm{Herm}^0_\mathbb F(V)$.
In \cref{prop:automorphism_group_action_on_traceless_hermitian_forms} we will show that its image actually lies in $\mathrm{SO}(p',q')$ for certain $p', q'$.

\subsection{Adjoint representation perspective}\label{sec:adjoint_representation_perspective}
An important observation that will be practical when proving the properties of the representations constructed in \cref{sec:adjoint_representations} is that $\mathrm{Herm}^0_\mathbb F(V)$ is a subspace of $\ssl(V, \mathbb F)$, and (by \cref{lem:identifications}) the automorphism group $\Aut(h_0)$ acts on it by conjugation,
which can lead to a different perspectives on the level of the Lie groups $\Aut(h_0), \GL(V, \mathbb F)$ and the Lie algebras $\mathrm{Lie}(\Aut(h_0)), \ssl(V, \mathbb F)$.
For clarity, we gather these observations in this section.

    Since the action of $G = \Aut(h_0) \cap \mathrm{SL}(V, \mathbb F) \subseteq \GL(V, \mathbb F)$ on $\mathrm{Herm}^0_\mathbb F(V)$ is given by conjugation, one can recover $\tau$ from the adjoint representation $\Ad: \GL(V, \mathbb F) \to \GL(\gl(V, \mathbb F))$.
    Indeed, by restricting first the domain of $\Ad$ from $\GL(d,\mathbb F)$ to $G$, and then each element $\Ad_g$ to the invariant subspace $\mathrm{Herm}^0_\mathbb F(V) \subseteq \gl(V, \mathbb F)$, one obtains $\tau$

    The same viewpoint can be pursued in the Lie algebra level, via the Cartan decomposition of $\ssl(n, \mathbb F)$ and the adjoint representation of the Lie algebra.
    More concretely, let $\mathfrak h = \ssl(n, \mathbb F)$, and consider the involutions $\theta_n(X) = -X^*$ and $\theta_0(X) = -X^{*_{h_0}}$ for $X \in \mathfrak h$, where $X^*$ is the adjoint of $X$ with respect to the standard Hermitian form $h_n$ on $\mathbb F^n$, and coincides with the conjugate transpose of the matrix $X$.
    Then the eigenspaces corresponding to the eigenvalue $1$ are given by the traceless parts of the Lie algebras of the automorphism groups:
    \begin{align*}
        \mathfrak h^{\theta_n} &= \mathrm{Lie}(\mathrm{Aut}(h_n)) \cap \ssl(n, \mathbb F) \\
        \mathfrak h^{\theta_0} &= \mathrm{Lie}(\mathrm{Aut}(h_0)) \cap \ssl(n, \mathbb F).
    \end{align*}
    denoting by $\mathfrak p_n = \mathfrak h^{-\theta_n}$, and $\mathfrak p_0 = \mathfrak h^{-\theta_0}$ the anti-fixed points of each involution, we observe that $\mathfrak p_n$ and $\mathfrak p_0$ are precisely the subspaces of traceless Hermitian forms if we use $h_n$ and $h_0$ as "reference" form respectively.

    The corresponding Cartan decompositions read:
    \begin{align*}
        \ssl(n, \mathbb F) &= \mathfrak g^{\theta_n} \oplus \mathfrak p_n \\
        &= \mathfrak g^{\theta_0} \oplus \mathfrak p_0.
    \end{align*}
    Looking closely - similarly to the situation at the level of a group - the action of $\mathrm{Lie}(\Aut(h_0))$ on $\mathrm{Herm}^0_\mathbb F(V)$ is given by the restriction of the adjoint representation $\ad: \gl(n, \mathbb F) \to \gl(\gl(n, \mathbb F))$ to the subalgebra $\mathrm{Lie}(\Aut(h_0)) \cap \ssl(n,\mathbb F) = \mathfrak g^{\theta_0}$ and the invariant subspace $\mathrm{Herm}^0_\mathbb F(V) = \mathfrak p_0$.

    Later it will be useful to refine one decomposition using the other by considering the composition of the two involutions.
    Given that they commute, the eigenspaces of their composition $\sigma = \theta_n \circ \theta_0$ are given by the intersections of the eigenspaces of $\theta_n$ and $\theta_0$, leading to the decomposition:
    \begin{align}\label{eq:refined_cartan_decomposition}
        \begin{split}
        \ssl(n, \mathbb F) = &(\mathfrak h^{\theta_n} \cap \mathfrak h^{\theta_0}) \oplus (\mathfrak p_n \cap \mathfrak p_0) \oplus \\
        &(\mathfrak h^{\theta_n} \cap \mathfrak p_0) \oplus (\mathfrak p_n \cap \mathfrak h^{\theta_0}),    
        \end{split}
    \end{align}
    with the first two summands corresponding to the eigenvalue $1$, and the last two to the eigenvalue $-1$ of~$\sigma$.

\subsection{Action on traceless Hermitian forms is strongly irreducible}\label{sec:strong_irreducibility}
In this section we will apply the perspective of \cref{sec:adjoint_representation_perspective} to show that the representation $\tau$ defined in \cref{def:representation_tau} is strongly irreducible when restricted to a uniform lattice.
The rest of its properties are presented and proved in \cref{sec:adjoint_representations}.

To recall, a representation $\rho: \Gamma \to \GL(V)$ is \emph{strongly irreducible} if the connected component of the identity of the Zariski-closure of $\rho(\Gamma)$ acts irreducibly on $V$.
In our case, since $\Gamma$ will be a uniform lattice of $G$, it suffices to show that the action of the Lie algebra $\mathfrak g$ of $G$ on $\mathrm{Herm}_\mathbb F^0 (\mathbb F^n)$ is irreducible (where $n = p + q$).
Indeed, then from the Borel density theorem, $\Gamma$ is Zariski-dense in $G$, so $\tau(\Gamma)$ is Zariski-dense in $\tau(G)$.
In particular $\overline{\tau(\Gamma)}_0 = \tau(G)_0$, so $\tau$ is strongly irreducible if and only if $\tau(G)_0$ acts irreducibly on $\mathrm{Herm}_\mathbb F^0 (\mathbb F^n)$.
From elementary Lie group theory, we know that this is equivalent to the irreducibility of the differential of $\tau_*: \mathfrak g \to \mathfrak{gl}(V)$ at the identity.

The following notation will be used for this subsection:
\begin{align*}
    \mathfrak h &= \mathfrak{sl}(n, \mathbb F) = \begin{cases}
    \left\{ X \in \mathfrak{gl}(n, \mathbb F) : \mathrm{tr}(X) = 0 \right\}, \text{ when } \mathbb F \neq \mathbb H,\\    
    \left\{ X \in \mathfrak{gl}(n, \mathbb F) : \mathrm{tr}_\mathbb R(X) = 0 \right\}, \text{ when } \mathbb F = \mathbb H\\
    \end{cases}\\
    \mathfrak g &= \mathrm{Lie}(\mathrm{Aut}(\tilde h_{p,q}) \cap \mathrm{SL}(n,\mathbb F)) = \left\{ X \in \mathfrak{sl}(n, \mathbb F) : X = - I_{p,q}X^* I_{p,q} \right\}\\
    \mathfrak p &= \mathrm{Herm}_0^\mathbb F(\mathbb F^n) = \left\{ X \in \mathfrak{sl}(n, \mathbb F) : X = I_{p,q}X^* I_{p,q} \right\}.
\end{align*}
We begin with the real and the quaternionic cases:
\begin{proposition}
    When $\mathbb F$ is $\mathbb R$ or $\mathbb H$, the adjoint representation of $\mathfrak{g}$ on $\mathfrak p$ is irreducible.
\end{proposition}
\begin{proof}
    Let $\mathbb F = \mathbb H$ and $\omega$ be the standard symplectic form on $\mathbb C^{2n}$.
    We will make use of the Lie algebra isomorphism $\phi$ and the Lie algebra hommomorphism $\psi$, given by
    \[
    \begin{array}{r@{\;}c@{\;}l@{\qquad}r@{\;}c@{\;}l}
        \phi: \mathfrak{gl}(2n, \mathbb C) &\to& \mathbb C^{2n} \otimes \mathbb C^{2n} & \psi: \mathbb C^{2n} \otimes \mathbb C^{2n} &\to& \mathbb C\\
        A &\mapsto&  \omega(A \cdot, \cdot), &
        h &\mapsto&  \sum_{i=1}^n h(e_i, e_{n+i}) - h(e_{n+i}, e_i).
    \end{array}
    \]
    
    Letting $\sigma$ be the involutive automorphism of $\mathfrak{sl}(n, \mathbb H)$ given by $\sigma(X) = -X^*$, we obtain the splitting into ($\pm 1$)-eigenspaces.
    After complexifying and applying $\phi$, we obtain splittings of $\mathfrak{sl}(2n,\mathbb C)$ and of $\mathrm{ker}(\psi)$ into representations of $\mathfrak {sp}(2n, \mathbb C)$:
    \begin{align*}
        \mathfrak sl(n,\mathbb H) &= \mathfrak{u}(n, \mathbb H) \oplus \mathfrak p,\\
        \mathfrak{sl}(2n, \mathbb C) &= \mathfrak{sp}(2n, \mathbb C) \oplus \mathfrak{sl}(2n, \mathbb C)^{-\sigma},\\
        \mathrm{ker}(\psi) &= \mathrm{Sym}^2(\mathbb C^{2n}) \oplus \wedge^2_0(\mathbb C^{2n}),
    \end{align*}
    where $\mathrm{Sym}^2(\mathbb C^{2n})$ are the symmetric 2-tensors of $\mathbb C^{2n}$ and $\wedge^2_0(\mathbb C^{2n})$ is the intersection of $\mathrm{ker}(\psi)$ with the antisymmetric 2-tensors.
    A standard result of representation theory (see for instance \cite[Theorem 17.5]{fulton2013representation}) tells us that $\wedge^2_0(\mathbb C^{2n})$ is an irreducible representation of $\mathfrak{sp}(2n, \mathbb C)$., which is equivalent to the irreducibility of $\mathfrak{sl}(2n, \mathbb C)^{-\sigma} \simeq \mathfrak {sl}(n, \mathbb H)^{-\sigma}$.
    This concludes the case of $\mathbb F = \mathbb H$.

    For the real case, we let $n = p+q$ and proceed analogously.
    The involution $\sigma$ is now defined as $\sigma(X) = -X^t$, and define
    \[
    \begin{array}{r@{\;}c@{\;}l@{\qquad}r@{\;}c@{\;}l}
        \phi: \mathfrak{gl}(n, \mathbb C) &\to& \mathbb C^{n} \otimes \mathbb C^{n} & \psi: \mathbb C^{n} \otimes \mathbb C^{n} &\to& \mathbb C\\
        A &\mapsto&  Q(A \cdot, \cdot), &
        h &\mapsto&  \sum_{i=1}^n h(e_i, e_{i}),
    \end{array}
    \]
    where $Q$ is the standard $\mathbb R$-linear inner product on $\mathbb C^n$.
    Through complexification and application of $\phi$, we obtain three splittings into representations of $\mathfrak{so(n, \mathbb C)}$:
    \begin{align*}
        \mathfrak{sl}(n, \mathbb R) &= \mathfrak{so}(p,q) \oplus \mathfrak p,\\
        \mathfrak{sl}(n, \mathbb C)&= \mathfrak{so}(n, \mathbb C) \oplus \mathfrak{sl}(n, \mathbb C)^{-\sigma},\\
        \mathrm{ker}(\psi) &= \wedge^2 (\mathbb C^{n}) \oplus \mathrm{Sym}^2_0(\mathbb C^n) ,
    \end{align*}
    where $\mathrm{Sym}^2_0(\mathbb C^n)$ is the intersection of the kernel of $\psi$ with the symmetric tensors.
    Then we conclude by the fact (see for instance \cite[Exercise 18.7]{fulton2013representation}) that $\mathrm{Sym}^2_0(\mathbb C^n)$ is an irreducible representation of $\mathfrak{so}(n,\mathbb C)$.
\end{proof}

For the complex case, the proof is even simpler, essentially because Hermitian and anti-Hermitian (identified with $\mathfrak g$) forms are isomorphic as representations, by multiplication by $i$ (the same is not true for the real and quaternionic cases, as can be checked by comparing their dimensions).
More specifically, in $\mathfrak h = \mathfrak{sl}(n, \mathbb C)$, the involution $\sigma(X) = -I_{p,q}X^* I_{p,q}$ is a complex conjugate-linear involution, making $\mathfrak g = \mathfrak{su}(p,q)$ a real form of $\mathfrak h = \mathfrak{sl}(n, \mathbb C)$, and $\mathrm{Herm}_\mathbb C^0 \mathbb C^n$ equal to $i \mathfrak{su}(p,q)$.
Thus $i$ provides us with an isomorphism between $\mathfrak{su}(p,q)$ and $\mathrm{Herm}_\mathbb C^0(\mathbb C^n)$ as vector subspaces of $\mathfrak{sl}(n, \mathbb C)$ and as representations of $\mathfrak{su}(p,q)$.
With this in mind, the irreducibility of the action of $\mathfrak{su}(p,q)$ on $\mathrm{Herm}_\mathbb C^0(\mathbb C^n)$ follows from the following proposition.
\begin{proposition}
    Let $\mathfrak h$ be a complex simple Lie algebra with real form $\mathfrak h_0$.
    Then the action of $\mathfrak h_0$ on $i\mathfrak h_0$ by the adjoint representation is irreducible. 
\end{proposition}
\begin{proof}
    If $\mathfrak h_1$ is an $\mathfrak h_0$-invariant subspace of $i\mathfrak h_0$, then $\mathfrak h_1 \oplus i\mathfrak h_1$ is an ideal of $\mathfrak h$.
    Since $\mathfrak h$ is simple, $\mathfrak h_1 \oplus i\mathfrak h_1$ is either trivial or the whole $\mathfrak h$, so $\mathfrak h_1$ is either trivial or the whole $i\mathfrak h_0$.
\end{proof}

\subsection{Setup for computations with Hermitian forms}\label{sec:setup}
In this subsection, we will make heavy use of the identifications of \cref{lem:identifications}, and to be consistent in what follows, we shall need to fix the setting in which we will work, i.e.\ the choice of basis and reference form.

\begin{setting}\label{set:standard_hermitian_form}
Let $n = p+q$, $V$ be an $n$-dimensional vector space over $\mathbb F \in \{ \mathbb R, \mathbb C, \mathbb H \}$, and $e_1, \cdots, e_{n}$ be a basis of $V$, which we use to identify it with $\mathbb F^{n}$.
We denote with $h_n \in \mathrm{Herm}_\mathbb F(V)$ the standard positive definite Hermitian form on $V \cong \mathbb F^n$, and with $h_{n-1,1}, \tilde h_{p,q} \in \mathrm{Herm}_\mathbb F(V)$ the nondegenerate Hermitian forms of signature $(n-1,1)$ and $(p,q)$ respectively.
Coordinatewise, these forms are given by
\[
\begin{array}{c}
h_{n}(x, y) = \sum_{i=1}^{n} \bar x_i y_i \\[6pt]
h_{n-1,1}(x,y) = \bar x_1 y_{n} + \bar x_{n} y_1 + \sum_{i=2}^{n-1} \bar x_i y_i \\[6pt]
\tilde h_{p,q}(x,y) = \sum_{i=1}^{p} \bar x_i y_i - \sum_{i = p+1}^{n}\bar x_{i} y_{i}\\
\end{array}
\text{ for } \begin{cases}
    x = \sum_i x_i e_i, y = \sum_i y_i e_i , \text{ if } \mathbb F \in \{\mathbb R, \mathbb C\}\\
    x = \sum_i e_i x_i, y = \sum_i e_i y_i, \text{ if } \mathbb F = \mathbb H
\end{cases}
\]
The corresponding matrices with respect to the basis $e_1, \cdots, e_{n}$ are given by
\[
H_n = I_{n}, \quad
h_{n-1,1} = 
\begin{pmatrix}
0 & 0 & 1 \\
0 & I_{n-2} & 0 \\
1 & 0 & 0
\end{pmatrix}
, \quad
\tilde h_{p,q} = 
\begin{pmatrix}
I_p& 0 \\
0 & -I_q
\end{pmatrix}.
\]
Using the basis $e_1, \cdots, e_{n}$ to identify $\gl(V)$ with $\mathbb F^{n \times n}$, we have the following identifications:
\begin{align*}
    \mathrm{Aut}(h_{n-1,1}) &= \left\{ g \in \mathbb F^{n, n} : h_{n-1,1} g^* h_{n-1,1} = g^{-1} \right\},\\
    \mathrm{Aut}(\tilde h_{p,q}) &= \left\{ g \in \mathbb F^{n, n} : \tilde h_{p,q} g^* \tilde h_{p,q} = g^{-1} \right\}.
\end{align*}
While automorphism groups of forms with the same signature are isomorphic, it will be more convenient at times to work with one rather than the other.
Depending on whose action we consider on $\mathrm{Herm}_\mathbb F(V)$, we will need to use the corresponding form to identify Hermitian forms and matrix representatives of self-adjoint transformations (see \cref{rem:reference_form_and_basis}).
More concretely:
\begin{enumerate}
    \item When (in \cref{sec:hyperbolic_spaces}) doing calculations with hyperbolic spaces $\mathbf H^n_\mathbb F$, we will work with $\mathrm{Aut}(h_{n-1,1})$, and identify $\mathrm{Herm}_\mathbb F(V)$ and the space of self-adjoint transformations with respect to $h_{n-1,1}$, i.e.\
    \begin{align}
        \begin{split}    
            \mathrm{Herm}_\mathbb F(V) &\cong \left\{ T \in \mathfrak{gl}(V): T^{*_{h_{n-1,1}}} = T \right\}\\
            &= \left\{ A \in \mathbb F^{n \times n} : h_{n-1,1} A^* h_{n-1,1} = A \right\}\\
            &= \left\{
                X = \begin{pmatrix}
                    X_{11} & X_{12} & X_{13}\\
                    X_{21} & X_{22} & X_{12}^*\\
                    X_{31} & X_{21}^* & \overline X_{11}
                \end{pmatrix} :
                \begin{array}{c}
            X_{11} \in \mathbb F, X_{13}, X_{31} \in \mathbb R,\\
            X_{12} \in \mathbb F^{1 \times (n-1)},
            X_{21} \in \mathbb F^{(n-1) \times 1},\\
            X_{22} \in \mathbb F^{(n-1) \times (n-1)},
            X_{22}^* = X_{22}
        \end{array}
        \right\}\label{eq:tilde_herm_n_1}
    \end{split}\\
    \begin{split}
        \mathrm{Lie}(\mathrm{Aut}(h_{n-1,1})) &= 
    \left\{ X \in \mathbb F^{n \times n}: X^{*_{h_{n-1,1}}} = - X \right\}\\
    = &\left\{
                X = \begin{pmatrix}
                    X_{11} & X_{12} & X_{13}\\
                    X_{21} & X_{22} & -X_{12}^*\\
                    X_{31} & -X_{21}^* & -\overline X_{11}
                \end{pmatrix} :
                \begin{array}{c}
            X_{11}, X_{13}, X_{31} \in \mathbb F, X_{12} \in \mathbb F^{1 \times (n-1)}, \\
            X_{21} \in \mathbb F^{(n-1) \times 1}, X_{22} \in \mathbb F^{(n-1) \times (n-1)}\\
            \overline X_{13} = -X_{13}, X_{22}^* = -X_{22}, \overline X_{31} = -X_{31}
        \end{array}
        \right\}.\notag
    \end{split}
    \end{align}
    This will be useful in \cref{lem:cartan_projections_automorphism_group}, \cref{ex:pozzetti}, and \cref{prop:limit_map_adjoint_representation}.
    \item When doing calculations on $\mathrm{Herm}_\mathbb F^0(V)$, we will work with $\mathrm{Aut}(\tilde h_{p,q})$, and identify $\mathrm{Herm}_\mathbb F(V)$ and the space of self-adjoint transformations with respect to $\tilde h_{p,q}$, i.e.\ 
    \begin{align*}
        \mathrm{Herm}_\mathbb F(V) &\cong\left\{ T \in \mathfrak{gl}(V): T^{*_{\tilde h_{p,q}}} = T \right\}\\
        &=\left\{ A \in \mathbb F^{n \times n} : \tilde h_{p,q} A^* \tilde h_{p,q} = A \right\}\\
        &= \left\{ X = \begin{pmatrix}
            A & B \\
            -B^* & D
            \end{pmatrix} :
            \begin{array}{c}
                A \in \mathbb F^{p \times p}, B \in \mathbb F^{p \times q}, D \in \mathbb F^{q \times q},
                A^* = A, D^* = D
            \end{array} \right\}.
    \end{align*}
    This will be useful in \cref{prop:automorphism_group_action_on_traceless_hermitian_forms}
    and \cref{prop:adjoint_representation_convex_cocompact}.
\end{enumerate}
\end{setting}

\subsection{Automorphism group action on Hermitian forms}\label{sec:adjoint_representations}
In this subsection we will be keep studying the action of the automorphism group of a non-degenerate Hermitian form on the space of traceless Hermitian forms.
Under the conventions of \cref{sec:setup}, we will show that it can give us $\mathbb H^{p,q}$-convex-cocompact representations of maximal virtual cohomological dimension that also contradict \cite[Lemma 6.8]{pozzetti_anosov_2023}.
As a first step, the next proposition tells us that the image of $\tau$ lies in fact in $\mathrm{SO}\left(\mathrm{Herm}^0_\mathbb F(V)\right)$, and recalls that it is strongly irreducible.
\begin{proposition}\label{prop:automorphism_group_action_on_traceless_hermitian_forms}
    Let $\tilde h_{p,q} \in \mathrm{Herm}_\mathbb F(V)$ be the standard non-degenerate Hermitian form of signature $(p,q)$ in the fixed basis of $V$:
    \[
    \tilde h_{p,q}\left( \sum_i x_i e_i, \sum_i y_j e_j \right) = \sum_{i=1}^p \bar x_i y_i - \sum_{i=p+1}^{n} \bar x_i y_i.
    \]
    Then the representation $\tau: G \to \GL\left(\mathrm{Herm}^0_\mathbb F(V)\right)$ defined in \cref{def:representation_tau} has the following properties:
    \begin{enumerate}
    \item Its restriction to any uniform lattice $\Gamma$ of $G$ is strongly irreducible.
    \item It preserves the inner product $\langle \cdot, \cdot \rangle$ on $\mathrm{Herm}^0_\mathbb F(V)$ defined as
    \[
    \langle A, B \rangle = -\tr_\mathbb R(AB),
    \]
    with signature
    \[
    \left(\dim_\mathbb R (\mathbb F) \cdot pq, \dim_\mathbb R (\mathbb F) \cdot \left(\frac{p(p-1)}{2} + \frac{q(q-1)}{2}\right) + p + q - 1\right),
    \]
    giving rise to a representation
    \[
    \tau: G \to \SO\left(\mathrm{Herm}^0_\mathbb F(V), \langle \cdot, \cdot \rangle\right),
    \]
    \end{enumerate}
\end{proposition}
\begin{proof}
    The proof of strong irreducibility was done in \cref{sec:strong_irreducibility}.
    Keeping the notation from \cref{sec:adjoint_representation_perspective}, we let $\mathfrak h = \ssl(n, \mathbb F)$, and consider the involutions $\theta_0(X) = -X^{*_{\tilde h_{p,q}}} = -I_{p,q} X^* I_{p,q}$ and $\theta_n(X) = -X^*$ for $X \in \mathfrak h$.
    denoting by $\mathfrak p_0 = \mathfrak h^{-\theta_0}, \mathfrak p_n = \mathfrak h^{-\theta_n}$ the anti-fixed points of each involution, we can compute explicitly that
    \begin{align}\label{eq:cartan_decomposition_hermitian_forms}
        \begin{split}
        \mathfrak p_0 &= \mathrm{Herm}^0_\mathbb F(V) = \left\{X \in \mathfrak{sl}(n,\mathbb F): X^{*_{\tilde h_{p,q}}} = X \right\}\\ 
        &=
         \left\{ X = \begin{pmatrix}
            A & B \\
            -B^* & D
            \end{pmatrix} :
            \begin{array}{c}
                A \in \mathbb F^{p \times p}, B \in \mathbb F^{p \times q}, D \in \mathbb F^{q \times q},\\
                A^* = A, D^* = D, \tr_\mathbb F(A) = - \tr_\mathbb F(D)
            \end{array}
            \right\}\\
        \mathfrak p_0 \cap \mathfrak h^{\theta_n} &= \mathrm{Herm}^0_\mathbb F(V) \cap \mathfrak h^{\theta_n} \\&= \left\{ X = \begin{pmatrix}
            0 & B \\
            -B^* & 0
            \end{pmatrix} : B \in \mathbb F^{p \times q} \right\}\\
        \mathfrak p_0 \cap \mathfrak p_n &= \mathrm{Herm}^0_\mathbb F(V) \cap \mathfrak p_n \\ &= \left\{ X = \begin{pmatrix}
            A & 0 \\
            0 & D
            \end{pmatrix} :
            \begin{array}{c}
                A \in \mathbb F^{p \times p}, D \in \mathbb F^{q \times q},\\
                A^* = A, D^* = D, \tr_\mathbb F(A) = - \tr_\mathbb F(D)
            \end{array} 
            \right\}    
        \end{split}
    \end{align}

    Since the inner product $\langle \cdot, \cdot \rangle$ is preserved by conjugation, $\rho(\mathrm{Aut}(\tilde h_{p,q}))$ is contained in the special orthogonal group $\mathrm{SO}\left(\mathrm{Herm}^0_\mathbb F(V), \langle \cdot, \cdot \rangle\right)$.
    To calculate its signature, we use the two summands of the Cartan decomposition of $\mathfrak g$ with respect to $\theta_0 \circ \theta_n$:
    \[
    \mathrm{Herm}_\mathbb F^0 (V) 
        = (\mathrm{Herm}_0^\mathbb F(V) \oplus \mathfrak p_n) \oplus \mathfrak (\mathrm{Herm}_0^\mathbb F(V) \cap \mathfrak h^{\theta_n}).
    \]
    It is straightforward to check that the restriction of $\langle \cdot, \cdot \rangle$ to $\mathrm{Herm}_0^\mathbb F(V) \cap \mathfrak h^{\theta_n}$ is positive definite, while its restriction to $\mathrm{Herm}_0^\mathbb F(V) \cap \mathfrak p_n$ is negative definite, and that they are orthogonal to each other.
    Hence we have that
    \begin{align*}
        \dim_\mathbb R \left( \mathrm{Herm}^0_\mathbb F(V) \cap \mathfrak h^{\theta_n} \right) &= 
        \dim_\mathbb R (\mathbb F) p q\\
        \dim_\mathbb R \left( \mathrm{Herm}^0_\mathbb F(V) \cap \mathfrak p_n \right) &= 
        \dim_\mathbb R(\mathbb F) \left( \frac{p(p-1)}{2} + \frac{q(q-1)}{2} \right) + p + q - 1
    \end{align*}
\end{proof}
The following proposition is the main result of this section, showing that $\tau$ admits an equivariant spacelike embedding of the symmetric space $G/K$ into a pseudo-Riemannian hyperbolic space $\mathbf H^{p',q'}$.
Moreover, the representation will be $\mathbf H^{p',q'}$-convex-cocompact exactly when $G$ is of rank 1, i.e.\ $\min\{p,q\} = 1$.
Note that here the restriction to the intersection $\mathrm{SL}(V) \cap \mathrm{Aut}(\tilde h_{p,q})$ is essential, because the embedding criterion of \cref{lem:criterion_spacelike_cocompact} requires that $G$ is semisimple. 
\begin{proposition}
    Let $V$ be a finite dimensional vector space over $\mathbb F \in \{ \mathbb R, \mathbb C, \mathbb H \}$, $\tilde h_{p,q} \in \mathrm{Herm}_\mathbb F(V)$ be a non-degenerate Hermitian form of signature $(p,q)$, and set
    \[
    p' = \dim_\mathbb R (\mathbb F) \cdot pq, \quad q' = \dim_\mathbb R (\mathbb F) \cdot \left(\frac{p(p-1)}{2} + \frac{q(q-1)}{2}\right) + p + q - 1.
    \]
    \begin{enumerate}
        \item The representation $\tau: G \to \SO\left(p', q'\right)$ defined in \cref{def:representation_tau}  admits a $G$-equivariant embedding $G/K \hookrightarrow \mathbf H^{p',q'-1}$ whose image $M^{p'}$ is an invariant complete connected spacelike embedded submanifold of dimension $p'$ in $\mathbf H^{p',q'-1}$.
        \item If moreover $\Gamma \subseteq G$ is a uniform lattice, then:
        \begin{enumerate}[label=(\roman*)]
        \item $\tau(\Gamma)$ acts cocompactly on $\tilde M^{p'}$.
        \item The restriction of $\tau$ to $\Gamma$ is $\mathbf H^{p',q'-1}$-convex-cocompact if and only if $\min\{p,q\} = 1$, in which case it is $\mathbf H^{p',q'-1}$-convex-cocompact of maximal virtual cohomological dimension.
    \end{enumerate}
    \end{enumerate}
\end{proposition}
\begin{proof}
    We begin by looking for a $\tau$-equivariant injection
    \[
    f : G/K \to \mathbf H^{p',q'-1},
    \]
    where $K$ is the maximal compact subgroup of $G$.
    Equivalently, we need a point $\mathbb Rx_0 \in \mathbf H^{p',q'-1}$ whose stabilizer in $G$ is $K$, since then we can take $f(gK) = \tau(g) \mathbb R x_0$, which is clearly well-defined, injective, and $\tau$-equivariant.

    Recall that the maximal compact subgroup of $G$ is given by
    \[
    K = \left\{
        \begin{pmatrix}
            A & 0 \\
            0 & D
        \end{pmatrix} : 
        \left.
            (A, D) \in
        \begin{cases}
            \mathrm{S}\left(\mathrm{O}(p) \times \mathrm{O}(q)\right), & \mathbb F = \mathbb R\\
            \mathrm{S}\left(\mathrm{U}(p) \times \mathrm{U}(q)\right), & \mathbb F = \mathbb C\\
            \mathrm{Sp}(p) \times \mathrm{Sp}(q), & \mathbb F = \mathbb H
        \end{cases}
        \right\} 
        \right\} \simeq
        \begin{cases}
            \mathrm{S}\left(\mathrm{O}(p) \times \mathrm{O}(q)\right), & \mathbb F = \mathbb R\\
            \mathrm{S}\left(\mathrm{U}(p) \times \mathrm{U}(q)\right), & \mathbb F = \mathbb C\\
            \mathrm{Sp}(p) \times \mathrm{Sp}(q), & \mathbb F = \mathbb H
        \end{cases}
    \]
    In particular, it coincides with the intersection of $G$ with the block diagonal matrices.
    We let
    \[
    x_0 = \begin{pmatrix}
        I_p & 0 \\
        0 & -\frac{p}{q}I_q
    \end{pmatrix} \in \mathrm{Herm}^0_\mathbb F(V),
    \]
    which satisfies $\langle x_0, x_0 \rangle = -p(p+q)/q < 0$, so $\mathbb R x_0 \in \mathbf H^{p', q'-1}$.
    It is clearly fixed by the action of $K$, since $\tau(k)x_0 = k x_0 k^{-1} = x_0$.
    On the other hand, if there exists some $\lambda \in \mathbb R$ such that $\tau(g)x_0 = \lambda x_0$, then $\lambda$ must be $1$, and $g$ must commute with $x_0$, which implies that $g \in K$.
    Thus $K = \mathrm{Stab}_G(\mathbb R x_0)$, and we can define $f(gK) = \tau(g) \mathbb R x_0$ as desired.

    To conclude that $f$ is an embedding, by \cref{lem:criterion_spacelike_cocompact} it suffices to show that it is spacelike (and thus an immersion).
    And because of $G$-equivariance, it suffices to check this only at a point.
    Indeed, 
    \begin{align*}
        \text{for } X = \begin{pmatrix}
            0 & B \\
            B^* & 0
        \end{pmatrix} \in \mathfrak p, \text{ we have }
        \d_K f(X + \mathfrak k) = \tau_*(X) x_0 = \left(1 + \frac{p}{q}\right) \begin{pmatrix}
            0 & -B \\
            B^* & 0
        \end{pmatrix}.
    \end{align*}
    Hence
    $
        \langle \d_K f(X + \mathfrak k), \d_K f(X + \mathfrak k) \rangle = 2 \left( 1 + \frac{p}{q} \right)^2 \tr(BB^*) \geq 0,
    $
    with equality if and only if $X = 0$.
    
    In case where $\tau$ is $\mathbf H^{p',q'-1}$-convex-cocompact, then by \cref{fact:convex_cocompact_hyperbolic}, $\Gamma$ is Gromov hyperbolic, which is equivalent to $G$ having rank $1$, i.e.\ $\min\{p,q\} = 1$. 
    If on the other hand $\min\{p,q\} = 1$, then $\Gamma$ is hyperbolic. Since we have shown that it acts properly discontinuously and cocompactly on a spacelike submanifold of $\mathbf H^{p', q'-1}$, \cite[Corollary 1.11]{beyrer2023} implies that it is $\mathbf H^{p',q'-1}$-convex-cocompact.
\end{proof}
\begin{remark}
For the rest of this subsection, we will switch the form $h_{n-1,1}$  for the coordinate representations of $\mathrm{Aut}(h_{n-1,1})$, $\mathbf H^{n-1}_\mathbb F$, and $\mathrm{Herm}_\mathbb F^0(V)$, which will be more practical for computations on the hyperbolic space.    
\end{remark}
Restricting to the rank-one case, we obtain a concrete description of the limit map and its tangent space.
To do this we need to consider three actions of the automorphism group: the one on the hyperbolic space $\mathbf H_\mathbb F^n$, the second on the space of traceless Hermitian forms $\mathrm{Herm}^0_\mathbb F(V)$ (which is identified with the space of self-adjoint transformations, as explained in \cref{sec:hermitian_forms_definitions}), and the third on its minimal parabolic subgroups.
The first is given by multiplication on the left, while the second and third are given by conjugation:
\begin{align*}
    g \cdot \mathbb F x &= \mathbb F g x,\\
    g \cdot X &= g X g^{-1},\\
    g \cdot P &= g P g^{-1},
\end{align*}
where $g \in G$, $\mathbb F x \in \mathbf H_\mathbb F^n$, $X \in \mathrm{Herm}^0_\mathbb F(V)$, and $P$ is a minimal parabolic subgroup of $G$.
\begin{proposition}\label{prop:limit_map_adjoint_representation}
    Consider the map
    \begin{align*}
        \xi^1: G/P &\to \mathbb P(\mathrm{Herm}^0_\mathbb F(V))\\
        gP &\mapsto \mathbb R g x_0 g^{-1},
    \end{align*}
    where
    \[
    x_0 = \begin{pmatrix}
        0 & 0 & \cdots & 0 & 1\\
        0 & 0 & \cdots & 0 & 0\\
        \vdots & \vdots & \ddots & \vdots & \vdots\\
        0 & 0 & \cdots & 0 & 0\\
        0 & 0 & \cdots & 0 & 0
    \end{pmatrix} \in \mathrm{Herm}^0_\mathbb F(V),
    \]
    and $P = \mathrm{Stab}_G([1:0: \cdots: 0])$ is the minimal parabolic group i.e.\ the stabilizer of $[1:0:\cdots:0] \in \partial_\infty \mathbf H^{n-1}_\mathbb F$ in $G$. 
    The following hold:
    \begin{enumerate}
    \item $\xi^1$ is smooth and $G$-equivariant.
    \item The tangent space of the image at the point $\xi^1(gP)$ is given by
    \[
    T_{\xi^1(gP)} \xi^1(\partial_\infty \Gamma) = \left( g [\mathfrak g, x_0] g^{-1} \right)/\xi^1(gP), \text{ for } x \in \partial_\infty \Gamma.
    \]
    \item If $\Gamma \leq G$ is a uniform lattice, then under the identification $\partial_\infty \Gamma \simeq G/P$, the map $\xi^1$ is the projective part of the limit map of $\tau|_\Gamma$.
    \end{enumerate}
\end{proposition}
\begin{proof}
    If we denote with $\mathfrak p$ the Lie algebra of the parabolic subgroup $P$, then we have that 
    \begin{align*}
        &\mathfrak p = \left\{X \in \mathfrak{sl}(n, \mathbb F) :  h_{n-1,1} X^* {h_{n-1,1}}^{-1} = - X, X \cdot [1: 0 : \cdots] = [1: 0 : \cdots]  \right\} =\\
        &\left\{
        \begin{pmatrix}
            X_{11} & X_{12} & X_{13}\\
            0 & X_{22} & -X_{12}^*\\
            0 & 0 & - \overline X_{11}
        \end{pmatrix} :
        \begin{array}{c}
            X_{11}, X_{13} \in \mathbb F, X_{12} \in \mathbb F^{1 \times (n-1)},
            X_{22} \in \mathbb F^{(n-1) \times (n-1)}\\
            \overline X_{13} = -X_{13}, X_{22}^* = -X_{22}, \tr_\mathbb F X_{22} = -2 \mathrm{Im} X_{11}
        \end{array}
        \right\}.
    \end{align*}
    To see this, consider for every $X \in \mathfrak p$, the path $p(t) = e^{tX}$ in $P$ which satisfies $p(0) = e$ and $p'(0) = X$.
    Since each $p(t)$ stabilizes $\mathbb F e_1$, there exists a map $\lambda: \mathbb R \to \mathbb F - \{0\}$ such that $p(t) e_1 = \lambda(t) e_1$ for every $t$.
    Then, $\lambda$ will be differentiable group homomorphism, so $\lambda(t) = e^{\alpha t}$ for some $\alpha \in \mathbb F^\times$, and in particular $X e_1 = \alpha e_1$.
    On the other hand, if $X e_1 = \alpha e_1$ for some $\alpha \in \mathbb F^\times$, then the one-parameter subgroup $e^{tX}$ preserves $[1:0:\cdots:0]$, so $X \in \mathfrak p$.
    
    Similarly, we see that $\tau(P)$ fixes the point $\mathbb R x_0$, since $\tau_*(\mathfrak p)$ fixes it.
    Thus the map $G \to \mathbb P(\mathrm{Herm}^0_\mathbb F(V))$ given by $g \mapsto \mathbb R g x_0 g^{-1}$ factors through the quotient $G/P$, giving us
    \begin{align*}
        \xi^1: G/P &\to \mathbb P(\mathrm{Herm}^0_\mathbb F(V))\\
        gP &\mapsto \mathbb R g x_0g^{-1}.
    \end{align*}
    
    It is immediate that the map $\xi^1$ is $G$-equivariant, so we obtain that the tangent space at the point $\xi^1(gP)$ is given by
    \[
    \d_{\xi^1(gP)} \xi^1(\partial_\infty \Gamma) = \pi( \tau(g)[\mathfrak g, x_0] ),
    \]
    where 
    \begin{align*}
    \pi: \mathrm{Herm_\mathbb F(V)} \to \mathrm{Herm}^0_\mathbb F(V)/\mathbb R \tau(g)x_0 \simeq T_{\xi^1(gP)} \mathbb P(\mathrm{Herm}^0_\mathbb F(V)).
    \end{align*}
\end{proof}
Given the above description of the limit map, we are ready to provide the counterexample to \cite[Lemma 6.8]{pozzetti_anosov_2023}.
Note that we do not include the case $\mathbb F = \mathbb R$ in the statement of the proposition, because then (as can be seen from \cref{eq:intersection_limit_spaces}), the intersection of the tangent spaces at two distinct points of the limit set will always be trivial.
\begin{proposition}\label{prop:pozzetti}
    Let $\mathbb F \in \{ \mathbb C, \mathbb H\}$, $G = \mathrm{Aut}(h_{n-1,1}) \cap \SL(n, \mathbb F)$, $\Gamma \leq G$ be a uniform lattice.
    Then the restriction $\tau: \Gamma \to \SO\left(\mathrm{Herm}^0_\mathbb F(V)\right)$ of the representation defined in \cref{def:representation_tau} is projective Anosov and strongly irreducible.
    
    Moreover, the map $\zeta: \partial_\infty \Gamma \simeq G/P \simeq \to \mathcal F_{1, p'}(\mathrm{Herm}_\mathbb F^0(V))$ given by
    \[
    \zeta(gP) = \left(\xi^1(gP), g [\mathfrak g, x_0] g^{-1} \right),
    \]
    with
    \[
    x_0 = \begin{pmatrix}
        0 & 0 & \cdots & 0 & 1\\
        0 & 0 & \cdots & 0 & 0\\
        \vdots & \vdots & \ddots & \vdots & \vdots\\
        0 & 0 & \cdots & 0 & 0\\
        0 & 0 & \cdots & 0 & 0
    \end{pmatrix} \in \mathrm{Herm}^0_\mathbb F(V),
    \]
    is a $\Gamma$-equivariant measurable section over the limit set $\xi^1(\partial_\infty \Gamma)$, and has the property that for every pair of points $gP, hP \in \partial_\infty \Gamma$, the intersection $\zeta^{p'}(gP) \cap \zeta^{p'}(hP)$ is not trivial.
    In particular, for every $(\rho(\Gamma), \phi)$-Patterson--Sullivan measure $\mu$ supported on $\zeta(\partial_\infty \Gamma)$, the representation $\rho$ is not $\mu$-irreducible.
\end{proposition}
\begin{proof}
    We consider the opposite minimal parabolic subgroups
    \[
    P = \mathrm{St}_G([1:0:\cdots:0]), P^t = \mathrm{St}_G([0: \cdots: 1]),
    \]
    and note that $P^t = g_0 P g_0^{-1}$, where
    \[
    g_0 = \begin{pmatrix}
        0 & \cdots & 0 & 1\\
        0 & \cdots & 1 & 0\\
        \vdots & \ddots & \vdots & \vdots\\
        0 & 1 & 0 & 0\\
        1 & 0 & 0 & 0
    \end{pmatrix} \in G.
    \]
    Since $G$ is of rank $1$, it acts transitively on pairs of distinct points in $G/P$, so for every pair $g \cdot P, h \cdot P \in G/P$ we can choose $g_1 \in G$ such that $g_1 g \cdot P = P$ and $g_1 h \cdot P = P^t$.
    By equivariance, we have that
    \begin{align*}
        \zeta^{p'}(gP) \cap \zeta^{p'}(hP) = \tau(g_1) \left( \zeta^{p'}(P) \cap \zeta^{p'}(P^t) \right),
    \end{align*}
    and thus $\zeta^{p'}(gP) \cap \zeta^{p'}(hP) \neq 0$ if and only if $\zeta^{p'}(P) \cap \zeta^{p'}(P^t) \neq 0$.
    We calculate
    \[
    \zeta^{p'}(P) = \tau_*(\mathfrak g) x_0 = \left\{ \mathbb R \begin{pmatrix}
        a & X_{12} & 2 b \\
        0 & 0 & X_{12}^*\\ 
        0 & 0 & -a
    \end{pmatrix}: a \in \mathbb F, b \in \mathbb R, X_{12} \in \mathbb F^{1 \times (n-1)}, \overline a = -a \right\}
    \]
    and
    \begin{align*}
        \zeta^{p'}(P^t) &= \tau(g_0) \zeta^{p'}(P) = g_0 \zeta^{p'}(P) g_0^{-1} \\
        &= \left\{ \mathbb R \begin{pmatrix}
        a & 0 & 0 \\
        X_{12}^* & 0 & 0\\ 
        2b & X_{12} & -a
    \end{pmatrix}: a \in \mathbb F, b \in \mathbb R, X_{12} \in \mathbb F^{1 \times (n-1)}, \overline a = -a \right\}
    \end{align*}
    giving that the intersection
    \begin{equation}
        \zeta^{p'}(P) \cap \zeta^{p'}(P^t) = \left\{ \mathbb R \begin{pmatrix}
        a & 0 & 0 \\
        0 & 0 & 0\\ 
        0 & 0 & -a
    \end{pmatrix}: a \in \mathbb F, \overline a = -a \right\}\label{eq:intersection_limit_spaces}
    \end{equation}
    is nontrivial.
\end{proof}
\subsection{Falconer and Hausdorff dimensions still coincide}\label{sec:critical_exponent}
As mentioned in \cref{sec:proof_strategy}, the example of the automorphism group action presented in the previous section does provide a counterexample to \cite[Lemma 6.8]{pozzetti_anosov_2023} and \cite[Proposition 10.3]{labourie_anosov_2006}, but it does not provide a counterexample to \cite[Theorem A]{pozzetti_anosov_2023}.
In other words, the Falconer exponent of the representation $\tau$ still coincides with the dimension of the limit set, and the objective of this section is to show this by direct calculation in \cref{sec:critical_exponent_calculation}.
The necessary facts about hyperbolic spaces and their volume entropy are recalled in \cref{sec:hyperbolic_spaces}.
\subsubsection{Hyperbolic spaces and their volume entropy}\label{sec:hyperbolic_spaces}
The hyperbolic space over the field $\mathbb F \in \{\mathbb R, \mathbb C, \mathbb H\}$ is defined as the set negative lines with respect to a Hermitian form of signature $(n,1)$ on $\mathbb F^{n+1}$:
\[
\mathbf H^n_\mathbb F = \left\{ z \in  \mathbb P_\mathbb F(\mathbb F^{n+1}) : \langle z, z \rangle < 0 \right\},
\]
where we use the notation $\mathbb P_\mathbb F(V)$ to denote the projective space of a vector space $V$ over a field $\mathbb F$, i.e.\ $\mathbb P_\mathbb F(V) = V/\mathbb F^\times$
For this section, we keep using the following form: 
\[
\langle z, w \rangle = h_{n,1}(z,w) = \bar z_1 w_{n+1} + \bar z_2 w_2 + \cdots + \bar z_n w_n + \bar z_{n+1} w_{1}, \text{ for } z, w \in \mathbb F^{n+1},
\]
because for it, the Cartan subalgebra $\mathfrak a$ of $G = \mathrm{Aut}(h_{n-1,1}) \cap \SL(n+1, \mathbb F)$ is given by the diagonal matrices of the form
\[
\mathfrak a = \left\{ \diag(a, 0, \cdots, 0, -a) : a\in \mathbb R \right\},
\]
which makes calculations easier.
We normalize the metric on $\mathbf H^n_\mathbb F$ so that the sectional curvature is contained in the interval $[-4, -1]$, giving us (see for instance \cite[Section 2.2]{parker2003notes}, \cite[Section 1]{kim2003geometry}) the following formula for the distance $d$ between two points $x, y \in \mathbf H^n_\mathbb F$:
\begin{align}\label{eq:distance_hyperbolic_space}
\cosh^2 d(x,y) &= \frac{\langle x, y \rangle \langle y, x \rangle}{\langle x, x \rangle \langle y, y \rangle}. 
\end{align}
In calculating the critical exponent of the unstable Jacobian in \cref{prop:critical_exponent_calculation}, we will use the critical exponent of the group $\Gamma$ acting on $\mathbf H^n_\mathbb F$, which coincides with the volume entropy of the space, since $\Gamma$ is a lattice acting cocompactly on a closed non-positively curved manifold.
Thankfully, the volume entropies for the rank-1 symmetric spaces of non-compact type are well-known (see e.g.\ \cite[Section 2]{besson1995entropies}):
\begin{proposition}
    Let $\mathbf H^n_\mathbb F$ denote the hyperbolic space over $\mathbb F = \mathbb R, \mathbb C, \mathbb H$ of $(\mathbb F-)$dimension $n$, with the metric normalized so that its sectional curvature is contained in the interval $[-4, -1]$.
    Then the volume entropy of $\mathbf H^n_\mathbb F$ is given by
    \[
    h(\mathbf H^n_\mathbb F) = \dim_\mathbb R(\mathbb F) n + \dim_\mathbb R(\mathbb F) - 2 = 
    \begin{cases}
        n-1 & \text{if } \mathbb F = \mathbb R,\\
        2n & \text{if } \mathbb F = \mathbb C,\\
        4n + 2 & \text{if } \mathbb F = \mathbb H.
    \end{cases}
    \]
\end{proposition}
\subsubsection{Calculation of the critical exponent}\label{sec:critical_exponent_calculation}
To go from the critical exponent of $\Gamma$ acting on $\mathbf H^n_\mathbb F$ to the critical exponent of the unstable Jacobian, we will need explicit expressions for the Cartan projections of $G = \mathrm{Aut}(h_{n,1}) \cap \SL(n+1, \mathbb F)$ and $\SO(\mathrm{Herm}_\mathbb F^0(V))$, in terms of the distance in $\mathbf H^{n}_\mathbb F$.
\begin{lemma}[Expressions for Cartan projections]\label{lem:cartan_projections_automorphism_group}
    Let $\tau: G \to \SO(\mathrm{Herm}_\mathbb F^0(V))$ be the representation defined in \cref{def:representation_tau}.
    We fix $o = [1:0:\cdots:-1]$ in $\mathbf H^{n}_\mathbb F$ and denote its stabilizer in $G$ by $K = \mathrm{Stab}_{G}(o)$, which is a maximal compact subgroup.
    Then:
    \begin{enumerate}[label=(\roman*)]
        \item For $\gamma \in G$, the Cartan projection $\mu: G \to \mathfrak a^+$ with respect to the decomposition $G = K \exp(\mathfrak a^+) K$ is given by
        \[
        \mu(\gamma) = d(\gamma \cdot o, o) H_0 \in \mathfrak a^+,
        \]
        where $\mathfrak a^+ = \mathbb R H_0$ is the Weyl chamber of $\mathfrak g$, and $H_0 = \diag(1, 0, \cdots, 0, -1)$.
        \label{item:cartan_projection_automorphism_group}
        \item There exists a maximal compact subgroup $K' \leq \SO(\mathrm{Herm}_\mathbb F^0(V))$ containing $\tau(K)$.
        \item The value $\tau_*(H_0)$ of the differential $\tau_*: \mathfrak g \to \mathfrak{so}((\mathrm{Herm}_\mathbb F^0(V)))$ of $\tau$ at the identity, is given by
        \begin{align*}
        \tau_*(H_0) = \diag(2, 1, \cdots, 1, 0, \cdots, 0, -1, \cdots, -1, -2) 
        \begin{array}{c}
            \text{ modulo the action} \\
            \text{ of the Weyl group}
        \end{array}
        \end{align*}

        where the entry $1$ appears $(\dim_\mathbb R F)(n-1)$ times, $0$ appears $\dim_\mathbb R \mathbb F + \dim_\mathbb R \mathfrak u(n-1, \mathbb F) - 1$ times, and $-1$ appears $(\dim_\mathbb R F)(n-1)$ times.
    \item For every $\gamma \in G$, we have that
    \begin{align*}
        \mu'(\tau(\gamma)) &= d(\gamma \cdot o, o) \diag(2, 1, \cdots, 1, 0, \cdots, 0, -1, \cdots, -1, -2)
    \end{align*}
    where $\mu': \SO(\mathrm{Herm}_\mathbb F^0(V)) \to \mathfrak a'^+$ is the Cartan projection of $\SO(\mathrm{Herm}_\mathbb F^0(V))$ with respect to the decomposition $\SO(\mathrm{Herm}_\mathbb F^0(V)) = K' \exp(\mathfrak a'^+) K'$, and 
    \begin{align*}
        \mathfrak a'^+ &= \left\{ \diag(a_1, \cdots, a_r, 0, \cdots, 0, -a_1, \cdots, -a_r) : a_1 \geq \cdots \geq a_r \geq 0 \right\}
    \end{align*}
    is the Weyl chamber of $\SO(\mathrm{Herm}_\mathbb F^0(V))$.
    \end{enumerate} 

\end{lemma}
\begin{proof}
    \begin{enumerate}[label=(\roman*)]
    \item Since $\mathfrak a^+ = \mathbb R_+ H_0$ for $H_0 = \diag(1, 0, \cdots, 0, -1)$, we can write $\mu(\gamma) = r(\gamma) H_0$ for some $r(\gamma) \in \mathbb R_+$, which we will show to be equal to the displacement $d(o, \gamma \cdot o)$ of $\gamma$ at the point $o$.
    Indeed, writing $\gamma = k e^{\mu(\gamma)} l$ for $k, l \in K$, we have that
    \begin{align*}
        d(\gamma \cdot & o, o) = d(k e^{\mu(\gamma)} l \cdot o, o)\\
        &= d(e^{\mu(\gamma)} \cdot [1:0:\cdots:-1], [1:0:\cdots:-1])\\
        &= d(\diag(e^{r(\gamma)}, \cdots, e^{-r(\gamma)})[1:0:\cdots:-1], [1:0:\cdots:-1])\\
        &= d([e^{r(\gamma)}:0:\cdots:0:-e^{-r(\gamma)}], [1:0:\cdots:-1])
    \end{align*}
    However, substituting into \cref{eq:distance_hyperbolic_space}, we have that
    \[
    \cosh^2\left(d([e^{r(\gamma)}:0:\cdots:0:-e^{-r(\gamma)}], o)\right) = \frac{\left(-e^{r(\gamma)} - e^{-r(\gamma)}\right)^2}{4} = \cosh^2(r(\gamma)),
    \]
    from which it follows that $r(\gamma) = d(\gamma \cdot o, o)$.

    \item We saw in \cref{sec:adjoint_representation_perspective} that we can use the involution $\theta_n(X) = -X^{*}$ to decompose $\mathrm{Herm}_\mathbb F^0(V)$ into
    \[
    \mathrm{Herm}_\mathbb F^0(V) = \left( \mathrm{Herm}_\mathbb F^0(V) \cap \mathfrak h^{\theta_n} \right) \oplus \left(\mathrm{Herm}_\mathbb F^0(V) \cap \mathfrak p_n\right),
    \]
    where $\mathfrak p_n$ are the anti-fixed points of $\theta_n$ and $\mathfrak h = \mathfrak{sl}(n,\mathbb F)$.
    Moreover, when endowed with the inner product $\langle X, Y \rangle = - \mathrm{Re}(\tr(XY))$, we can calculate that the decomposition is orthogonal, 
    its first summand is positive definite, and the second one is negative definite.
    Moreover, each summand is invariant
     under the action of $\rho(K)$, so 
    \[
    \rho(K) \subseteq K' \equaldef \mathrm S \left(\mathrm{O}(\mathrm{Herm}_\mathbb F^0(V) \cap \mathfrak g^{\theta_n}) \times \mathrm{O}(\mathrm{Herm}_\mathbb F^0(V) \cap \mathfrak p_n)\right).
    \]
    \item
    
    Since $G$ acts on $\mathrm{Herm}_\mathbb F^0(V)$ by conjugation, its differential at the identity acts by the commutator, i.e.\ $\tau_*(X) Y = [X, Y]$ for $X \in \mathfrak g$ and $Y \in \mathrm{Herm}_\mathbb F^0(V)$.
    We can thus calculate that for $X \in \mathrm{Herm}_\mathbb F^0(V)$,
    \[
    \tau_*(H)(X) = [H, X] = \begin{pmatrix}
        0 & a X_{12} & 2 a X_{13}\\
        -a X_{21} & 0 & a X_{12}^*\\
        -2 a X_{31} & -a X_{21}^* & 0
    \end{pmatrix}, 
    \begin{array}{c}
        \text{ for } H = \diag(a, 0, \cdots, 0, -a)\\
        \text{ and } X \text{ as in } \cref{eq:tilde_herm_n_1}
    \end{array}
    \]
    Thus the weights of the representation $\rho$ with respect to the Cartan subalgebra $\mathfrak a$ of diagonal matrices are given in \cref{tab:weightspaces_automorphism_group}, where we use the block notation from the identification stated in \cref{eq:tilde_herm_n_1}, and denote with $L_1 \in \mathfrak a^*$ the projection of the first diagonal entry.
    From this, we have that
    \[
    \mu'(\tau_*(H_0)) = \diag(2, 1, \cdots, 1, 0, \cdots, 0, -1, \cdots, -1, -2) \in \mathfrak a'^+
    \]
    with each entry appearing as many times as its corresponding weightspace dimension.
    \begin{table}[ht]
        \centering
        \caption{Weightspaces and their dimensions}
        \label{tab:weightspaces_automorphism_group}
        \begin{tabular}{|c|c|c|}
            \hline
            Weight & Weightspace corresponding to & dimension \\ \hline
            $2L_1$ & $X_{13}$ & $1$ \\ \hline
            $-2L_1$ & $X_{31}$ & $1$ \\ \hline
            $L_1$ & $X_{12}$ & $\dim_\mathbb R F (n-1)$ \\ \hline
            $-L_1$ & $X_{21}$ & $\dim_\mathbb R F (n-1)$ \\ \hline
            $0$ & $X_{11}, X_{22}$ & $\dim_\mathbb R \mathbb F + \dim_\mathbb R \mathfrak u(n-1, \mathbb F) - 1$ \\ \hline
        \end{tabular}
    \end{table}
    \item
    To get the expression for any element $\tau(\gamma)$, we consider the Cartan decomposition of $\gamma = k e^{\mu(\gamma)} l$ for $k, l \in K$.
    Applying $\tau$ we obtain $\tau(\gamma) = \tau(k) e^{\tau*(\mu(\gamma))} \tau(l)$ and the result then follows by part \ref{item:cartan_projection_automorphism_group}.
    \end{enumerate}
    
\end{proof}

We are now ready to show that $\tau$ does not provide a counterexample to the equality between the critical exponent and the Hausdorff dimension.
\begin{proposition}[Equality of Falconer and Hausdorff dimensions]\label{prop:critical_exponent_calculation}
    Let $\Gamma \leq G$ be a uniform lattice, and $\tau: \Gamma \to \SO(\mathrm{Herm}^0_\mathbb F(V))$ be the representation defined in \cref{def:representation_tau}.
    Then the critical exponent of the Falconer functional is the same as the 
    dimension of the limit set~$\xi^1(\partial_\infty \Gamma)$:
    \[
    \delta_F(\tau) = \dim (\xi^1(\partial_\infty \Gamma)) = (\dim_\mathbb R \mathbb F) n-1.
    \]
\end{proposition}
\begin{proof}
    For the second equality, it suffices to note that $\Gamma$ is a uniform lattice, so the limit set $\xi^1(\partial_\infty \Gamma)$ is a smooth sphere of dimension $\dim(\xi^1(\partial_\infty \Gamma)) = \dim(\partial_\infty \Gamma) = \dim_\mathbb R(\mathbf H^n_\mathbb F) - 1 = (\dim_\mathbb R \mathbb F) n - 1$.

    As we recall from \cref{sec:proof_strategy}, in \cite{pozzetti_anosov_2023}, the authors show that when the limit set is a Lipschitz sphere, the critical exponent of the Falconer functional will be equal to the Hausdorff dimension of the limit set if and only if the critical exponent of the unstable Jacobian $J^u_{d_\Gamma}$ is equal to $1$, where $d_\Gamma = \dim(\partial_\infty \Gamma)$.
    To show that this holds in our example, we evaluate the unstable Jacobian over $\mu'(\tau(\gamma))$.
    For the real case, we have 
    \[
    J^u_{n-1}(\mu'(\tau(\gamma))) = d(\gamma \cdot o, o) (n-1)(2-1) =(n-1) d(\gamma \cdot o, o).
    \]
    Hence $\delta_{J_{n-1}^u}(\tau) = \delta_\Gamma/(n-1) = 1$, where $\delta_\Gamma$ the classical critical exponent of $\Gamma$, as it appears in \cite{sullivan1979density}, i.e.\ 
    \begin{align*}
        \delta_\Gamma = \lim_{n \to \infty} \frac{\log \left( \sharp \left\{ \gamma \in \Gamma : d(\gamma \cdot o, o) \leq n \right\}\right)}{n}.
    \end{align*}
    For the complex case:
    \begin{align*}
        J^u_{2n-1}(\mu'(\tau(\gamma))) &= d(\gamma \cdot o, o) ( \alpha_{1,2} + \cdots + \alpha_{1,2n}) (\tau_*(H_0))\\
        &= 
        d(\gamma \cdot o, o) \cdot \left( (2n - 2) (2-1) + (2-0) \right)\\
        &= 2n\ d(\gamma \cdot o, o).
    \end{align*}
    so we have that
    \begin{align*}
        \delta_{J^u_{2n-1}}(\tau) = \frac{\delta_\Gamma}{2n} = 1.
    \end{align*}
    And finally, for the quaternionic case:
    \begin{align*}
        J^u_{4n-1}(\mu'(\tau(\gamma))) &= d(\gamma \cdot o, o) (\alpha_{1,2} + \cdots + \alpha_{1,4n}) (\tau_*(H_0))\\
        &=
        d(\gamma \cdot o, o) \cdot \left( (4n-4)(2-1) + 3(2-0) \right)\\
        &= (4n+2)\  d(\gamma \cdot o, o),
    \end{align*}
    implying that $\delta_{J_{4n-1}^u}(\tau) = 1$ as well.
\end{proof}

\section[\texorpdfstring{The action of $\mathrm{SO}(p,p)$ on maximal isotropic subspaces}{The action of SO(p,p) on maximal isotropic subspaces}]{The action of $\SO(p,p)$ on maximal isotropic subspaces}\label{sec:so_p_p_action}
While the representations of \cref{sec:adjoint_representations} already show that the result of \cite[Proposition 10.3]{labourie_anosov_2006} does not hold, in this section we give a more direct counterexample, which is interesting in its own right.
With it we show that one cannot always use the action of $\SO(p,p)$ to move a pair of maximal isotropic subspaces of $\mathbb R^{p,p}$ in transverse position.
This contradicts the claim of the proposition, because the action of the identity component $\SO_0(p,p)$ on $\mathbb R^{2p}$ is irreducible.

We recall that $\mathbb R^{p,p}$ is the vector space $\mathbb R^{2p}$ equipped with the symmetric bilinear form of signature $(p,p)$ given by
\[
\langle x, y \rangle_{p,p} = \sum_{i=1}^p x_i y_{i} - \sum_{i=1}^p x_{p+i} y_{p+i}.
\]
\labourie
\begin{proof}
    Prior to making any assumptions on the parity of $p$, we parametrize the set $S$ by $\mathrm{O}(p)$.
    \[\begin{array}{lll}
        &S \quad &\leftrightarrow \quad \mathrm{O}(p) \\
        &V \quad &\mapsto \quad \mathrm{pr}_2 \circ ({\mathrm{pr}_1}_{|V})^{-1} \\
        &\Gamma(\phi) \quad &\mapsfrom \quad \phi
    \end{array}\]
    where we denote with $\mathrm{pr}_1, \mathrm{pr}_2: \mathbb R^{p,p} \to \mathbb R^p$ the projections to the first and last $p$ coordinates, and $\Gamma(\phi) = \left\{ (x, \phi(x)) : x \in \mathbb R^p\right\}$ the graph of $\phi \in \mathrm{O}(p)$. 
    The left-to-right direction of this correspondence is well-defined because all $V \in S$ are $p$-dimensional and transverse to $\mathbb R^p \oplus 0$ and $0 \oplus \mathbb R^p$.
    Indeed, if we have ${\mathrm{pr}_1}_{|V}(x,y) = {\mathrm{pr}_1}_{|V}(x, y')$ for some $(x,y), (x,y') \in V$, then $(0, y-y') \in V \cap (0 \oplus \mathbb R^p)$, which implies $y = y'$.
    In particular, $S$ has two connected components, consisting of the graphs of elements of $\mathrm{O}(p)$ with positive and negative determinant respectively:
    \[
    S_\pm = \left\{ \Gamma(\phi) : \phi \in \mathrm{O}(p), \det \phi = \pm 1\right\}.
    \]

    Using the parametrisation above, we have a convenient way of checking whether two isotropic subspaces are transverse.
    Namely, for $\phi, \psi \in \mathrm{O}(p)$ we have that $\Gamma(\phi)$ and $\Gamma(\psi)$ are transverse if and only if $\phi \circ \psi^{-1}$ does not admit $1$ as an eigenvalue.
    However, because $\phi \circ \psi^{-1} \in \OO(p)$, depending on the parity of $p$, we may infer from the sign of its determinant that $1$ is an eigenvalue.
    denoting by $\sigma_-$ the (possibly 0) multiplicity of $-1$ as an eigenvalue of $\phi \circ \psi^{-1}$, we have that 
    \[
    \mathrm{det}(\phi) = (-1)^{\sigma_-} \cdot \mathrm{det}(\psi),
    \]
    so that
    \[
    \Gamma(\phi) \cap \Gamma(\psi) \neq 0 \text{ provided that } 
    \begin{cases}
        \det \phi = \det \psi & \text{ if } p \text{ is odd},\\
        \det \phi = -\det \psi & \text{ if } p \text{ is even}.
    \end{cases}
    \]

    Given this, the non-transversality of $gV$ and $W$ for every $g \in \SO(p,p)$ will follow as soon as we show that the action of $\SO(p,p)$ preserves the components of $S$.
    Being a Lie group with finitely many connected components, $\SO(p,p)$ deformation retracts to its maximal compact subgroup $S(\mathrm O (p) \times \mathrm O(p))$, which preserves the components of $S$ because its action is given by $(g, h) \Gamma(\phi) = \Gamma(h \phi g^{-1})$.
    Hence $\SO(p,p)$ preserves the components of $S$ as well.
\end{proof}
For a proof that a Lie group $G$ with finitely many components deformation retracts to its maximal compact subgroup $K$, we refer the reader to \cite[14.3.11]{hilgert2011structure}.
But in the case of a semisimple Lie group like $\SO(p,p)$, one can also see this directly from the Cartan decomposition $G \simeq K \times \mathfrak p$.

\printbibliography
\end{document}

%% file: packages.tex
\usepackage[T1]{fontenc}
\usepackage[utf8]{inputenc}

\usepackage[a4paper]{geometry}
\usepackage{graphicx}
\usepackage{subcaption}

\usepackage{amssymb,amsthm, amsmath}

\usepackage[
natbib,
style=alphabetic,
maxbibnames=10,  
sorting=nyt  ,
url=false,
doi=false,
sortcites,
defernumbers,
backref,
backend=biber
]{biblatex}
\usepackage{hyperref}

\usepackage[marginpar]{todo}

\usepackage{url}
\usepackage{tikz-cd}
\usepackage[nameinlink, capitalise, noabbrev]{cleveref}

\usepackage{xfrac}
\usepackage{nicefrac}

\usepackage{soul}

\usepackage{bbm}

\usepackage{enumitem}

\usepackage{stmaryrd}

\usepackage{thmtools, thm-restate}

\usepackage{mathdots}

\crefname{assumption}{Assumption}{Assumptions}

%% file: commands.tex
\newtheorem{theorem}{Theorem}[section]

\newtheorem{proposition}[theorem]{Proposition}

\newtheorem{vagueconjecture}{Vague Conjecture}
\newtheorem{lemma}[theorem]{Lemma}
\newtheorem{corollary}[theorem]{Corollary}
\theoremstyle{remark}
\newtheorem{remark}[theorem]{Remark}

\theoremstyle{definition}

\newtheorem{definition}[theorem]{Definition}

\newtheorem{fact}[theorem]{Fact}
\newtheorem*{setting}{Setting}

\crefname{cond}{Condition}{Conditions}
\crefformat{cond}{Condition~(#2#1#3)}
\crefname{ineq}{Inequality}{Inequalities}
\crefformat{ineq}{Inequality~(#2#1#3)}

\renewcommand{\d}{\,\mathrm{d}}

\newcommand{\defeq}{:=}

\DeclareMathOperator{\diag}{diag}

\DeclareMathOperator{\OO}{{\mathrm{O}}}
\DeclareMathOperator{\SL}{{\mathrm{SL}}}
\DeclareMathOperator{\ssl}{{\mathfrak{sl}}}

\DeclareMathOperator{\PSL}{{\mathrm{PSL}}}
\DeclareMathOperator{\GL}{{\mathrm{GL}}}
\DeclareMathOperator{\gl}{{\mathfrak{gl}}}
\DeclareMathOperator{\SO}{{\mathrm{SO}}}

\DeclareMathOperator{\PSO}{{\mathrm{PSO}}}

\DeclareMathOperator{\Ad}{Ad}
\DeclareMathOperator{\ad}{ad}
\DeclareMathOperator{\Aut}{Aut}

\DeclareMathOperator{\supp}{supp}

\DeclareMathOperator{\tr}{tr}

\newcommand{\equaldef}{\overset{\mathrm{def}}{=}}
